\documentclass[a4paper,11pt, reqno]{amsart}
\usepackage{a4wide}
\usepackage{amsmath}
\usepackage{amsthm}
\usepackage{amsfonts}
\usepackage{amssymb}
\usepackage{wrapfig}
\usepackage{epsfig}
\usepackage{graphicx}
\usepackage[english]{babel}
\usepackage{float}
\usepackage{amsthm}
\usepackage{enumerate}
\usepackage{caption}
\usepackage{subcaption}
\usepackage{subcaption}
\usepackage[utf8]{inputenc}
\usepackage{dsfont}
\usepackage{listings}
\usepackage{bbm}
\usepackage[tmargin=1.3in,bmargin=1.3in,lmargin=1.1in,rmargin=1.1in]{geometry}

\graphicspath{{images/}}

\theoremstyle{plain}
\newtheorem{theorem}{Theorem}[section]
\newtheorem{lemma}[theorem]{Lemma}
\newtheorem{proposition}[theorem]{Proposition}

\newtheorem{definition}[theorem]{Definition}

\theoremstyle{definition}

\usepackage{hyperref}

\newcommand{\E}{\mathbb{E}}

\newcommand{\Prob}{\mathbb{P}}
\renewcommand{\P}{\mathbb{P}}
\newcommand{\Zd}{\mathbb{Z}^d}

\newcommand{\Qcal}{\mathcal{Q}}
\newcommand{\Ccal}{\mathcal{C}}
\DeclareMathOperator*{\argmax}{arg\,max}

\numberwithin{equation}{section}

\title{Structures in supercritical scale-free percolation}
\author{Markus Heydenreich}
\address{Mathematisches Institut\\ Ludwig-Maximilians-Universit\"at M\"unchen\\ Theresienstra{\ss}e 39\\ D-80333 M\"unchen\\ Germany}
\email{m.heydenreich@lmu.de}
\author{Tim Hulshof}
\author{Joost Jorritsma}
\address{Department of Mathematics and Computer Science\\ Eindhoven University of Technology\\ P.O. Box 513\\ 5600 MB Eindhoven\\ The Netherlands}
\email{w.j.t.hulshof@tue.nl, j.jorritsma@student.tue.nl}

\date{\today}

\begin{document}
\begin{abstract}
Scale-free percolation is a percolation model on $\Zd$ which can be used to model real-world networks. We prove bounds for the graph distance in the regime where vertices have infinite degrees. We fully characterize transience vs.\ recurrence for dimension 1 and 2 and give sufficient conditions for transience in dimension 3 and higher. Finally, we show the existence of a hierarchical structure for parameters where vertices have degrees with infinite variance and obtain bounds on the cluster density.
\end{abstract}

\maketitle

\section{Introduction}
Random graphs are mathematical models commonly used to study real-world networks such as the World-Wide Web, social, financial, neural, and biological networks. 
Many real-world networks exhibit the following two properties:
		\begin{itemize}
			\item \emph{The small-world property:} distances within the network are very small in comparison to the number of nodes. With ``small'' we mean distances are at most of order of an iterated logarithm.
			Some real-world networks are even \emph{ultra-small,} meaning that the distances are at most a double logarithm.
			\item \emph{The scale-free property:} the number of connections per node behave statistically like a power-law. This implies that the variation is typically very high.			
		\end{itemize}
An example of a random graph model with these properties is the \emph{Norros-Reittu random graph} \cite{norros2006} (see Figure \ref{FigureNorros}). This model produces a random graph $G=(V,E)$ on a fixed set of vertices $V$, but with a random edge set $E\subset V\times V$ as follows: Every vertex $x\in V$ is assigned an i.i.d.\ random weight $W_x>0$.  Conditioned on the weights of its end-vertices, the edge $\{x,y\}$ is present in $E$ with probability $p_{xy}=1-\exp(W_xW_y/\mathcal{N})$, independently of the status of other possible edges (here $\mathcal{N}$ is a normalizing constant). See \cite{RGN} for more results on inhomogeneous random graphs.  	
		 
		These two properties are important, but the structure of many real-life networks, such as social networks, often have other features that influence the structure and formation of networks: 
		\begin{itemize} 
		\item \emph{Geometric clustering:} in social networks this manifests itself because people who are geographically close to each other are more likely to know each other, giving rise to formation of locally concentrated clusters within the network.
		\item \emph{Hierarchies:} again in social networks, the more `important' people are, the more likely they know other important people, even if those people might be far away, giving rise to hierarchies within the network.
		\end{itemize}
		
		A well-known model that has geometric clustering and the connections over long distances required for the existence of hierarchies is \emph{long-range percolation} (LRP, see Figure \ref{FigureLRP}) \cite{Benjamini, BergerTransience,biskup2004, HeyHofSak08, Schulman}. LRP is a percolation model that produces random subgraphs of the graph $(\Zd, \Zd \times \Zd)$ wherein an edge $\{x,y\}\in\Zd \times \Zd$ is (independently) \emph{retained} with probability $p_{xy} \propto \lambda/|x-y|^\alpha$ for some positive constants $\lambda$ and $\alpha$, and removed otherwise. Thus, the connection probabilities are monotonically decreasing in $\alpha$, and increasing in $\lambda$. For many choices of $d$ and $\alpha$, LRP has a \emph{percolation phase transition} in $\lambda$, meaning that there exists $\lambda_c(d, \alpha) \in (0,\infty)$ such that when $\lambda > \lambda_c$ there exists an infinite cluster almost surely, whereas when $\lambda < \lambda_c$, all clusters are almost surely finite. When $\alpha\in (d,2d)$, this model has the clustering property, as well as something akin to the small-world property \cite{biskup2004}. It is, however, clearly not scale-free, since the decay of the degree distribution is faster than exponential.  

		Various models have been introduced in the recent years that combine three or four of the network properties described above. We mention, for instance, the models introduced by Aiello \emph{et al.}\ \cite{aiello2008},  Flaxman, Frieze, and Vera \cite{flaxman2006},  and Jacob and M\"orters \cite{jacob2013}. 
		
		In this paper we consider another model that has all four properties: \emph{scale-free percolation} (SFP, also known as \emph{heterogeneous long-range percolation}). SFP interpolates between long-range percolation and the Norros-Reittu random graph (see Figure \ref{FigureScalefree}). SFP was introduced by Deijfen, van der Hofstad, and Hooghiemstra in \cite{DeijfenScaleFree}. We start with a formal definition of the model.
		 
		\begin{definition}[Scale-free percolation]\label{SFPDef}
					Consider the graph $(\Zd, \Zd \times \Zd)$ for some fixed $d\geq1$. Assign to each vertex $x\in\Zd$ an i.i.d.\ weight $W_x$, where the weights follow a power-law distribution with parameter $\tau-1$:
					\begin{equation*}\label{powerlawDistribution}
					\Prob(W_x>w)= w^{-(\tau-1)}L(w), \qquad w > 0, 
					\end{equation*}
					where $L$ is a slowly-varying function (i.e., $L(wa)/L(w)\rightarrow1$ for all $a>0$ as $w\rightarrow\infty$, so the law of $W_x$ is $(\tau-1)$-regularly varying).
					Conditionally on the weights, an edge $\{x,y\}\in\Zd\times\Zd$ is retained independently of all other edges with probability 
					\begin{equation*}
					p_{xy}=1-\exp\left(-\lambda \frac{W_xW_y}{\vert x-y\vert^\alpha}\right),
					\end{equation*}
					where $\vert x\vert=\| x \|_1$ and $\lambda, \alpha>0$ are positive constants of the model.\footnote{We choose to work with the $\ell_1$-norm because it is a practical metric, but defining SFP with respect to any $\ell_p$-norm with $p \in [1,\infty]$ gives qualitatively similar results.} The edge is removed otherwise. We call retained edges \emph{open}, and removed edges \emph{closed.}
We denote the joint probability measure of edge occupation and weights by $\P_{(\lambda,W)}$ (where the subscript $W$ refers to the \emph{law} of the weights, not the actual values) and write just $\P$ if the parameters are clear from the context. 			
		\end{definition}
		
		\begin{figure}
		\begin{subfigure}{.48\linewidth}
		\includegraphics[keepaspectratio = True,width = 0.99\linewidth]{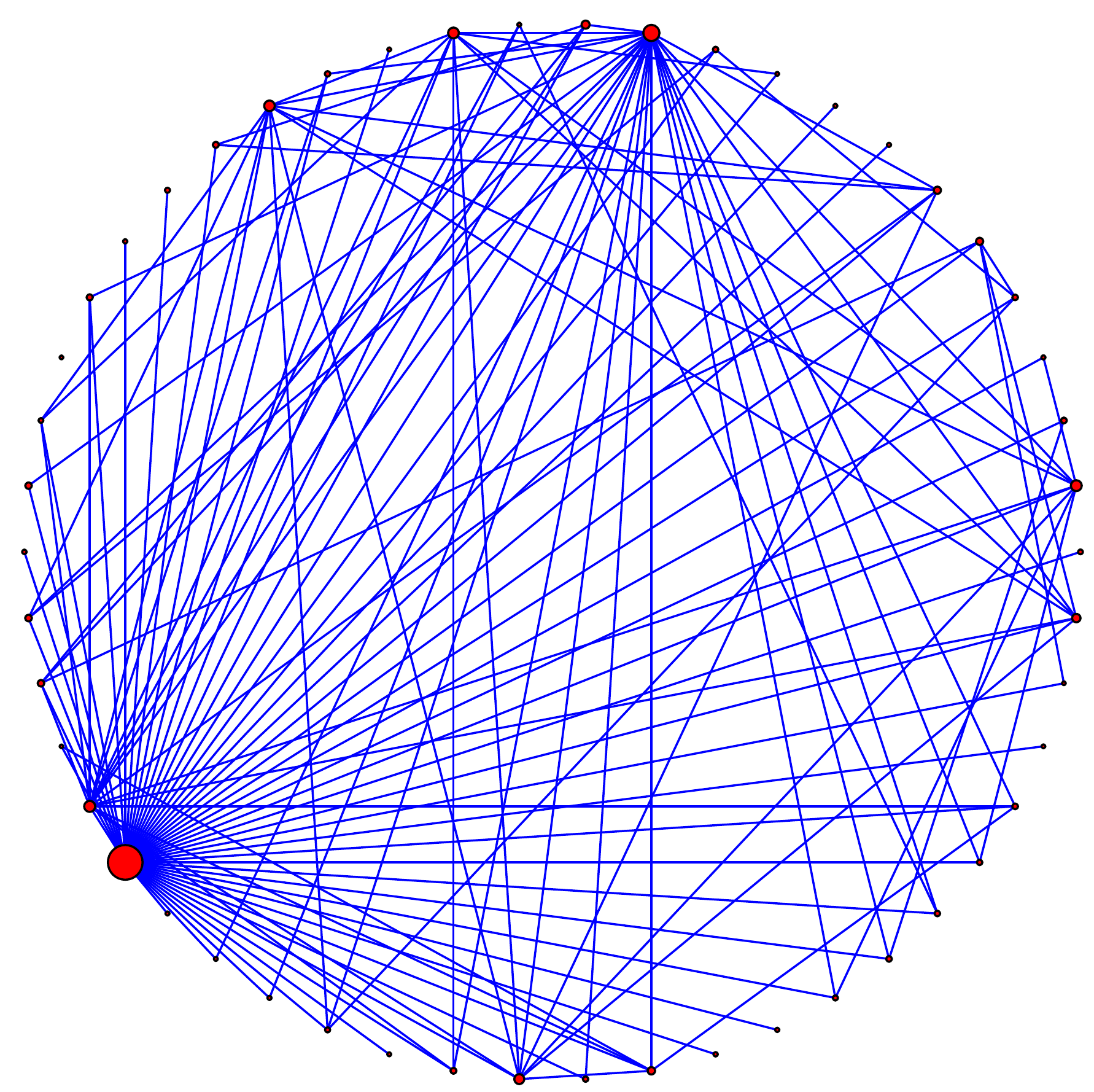}
		\caption{Norros-Reittu random graph for $\tau = 1.95$}\label{FigureNorros}
		\end{subfigure}	
		\begin{subfigure}{.48\linewidth}
		\includegraphics[keepaspectratio = True,width = 0.99\linewidth]{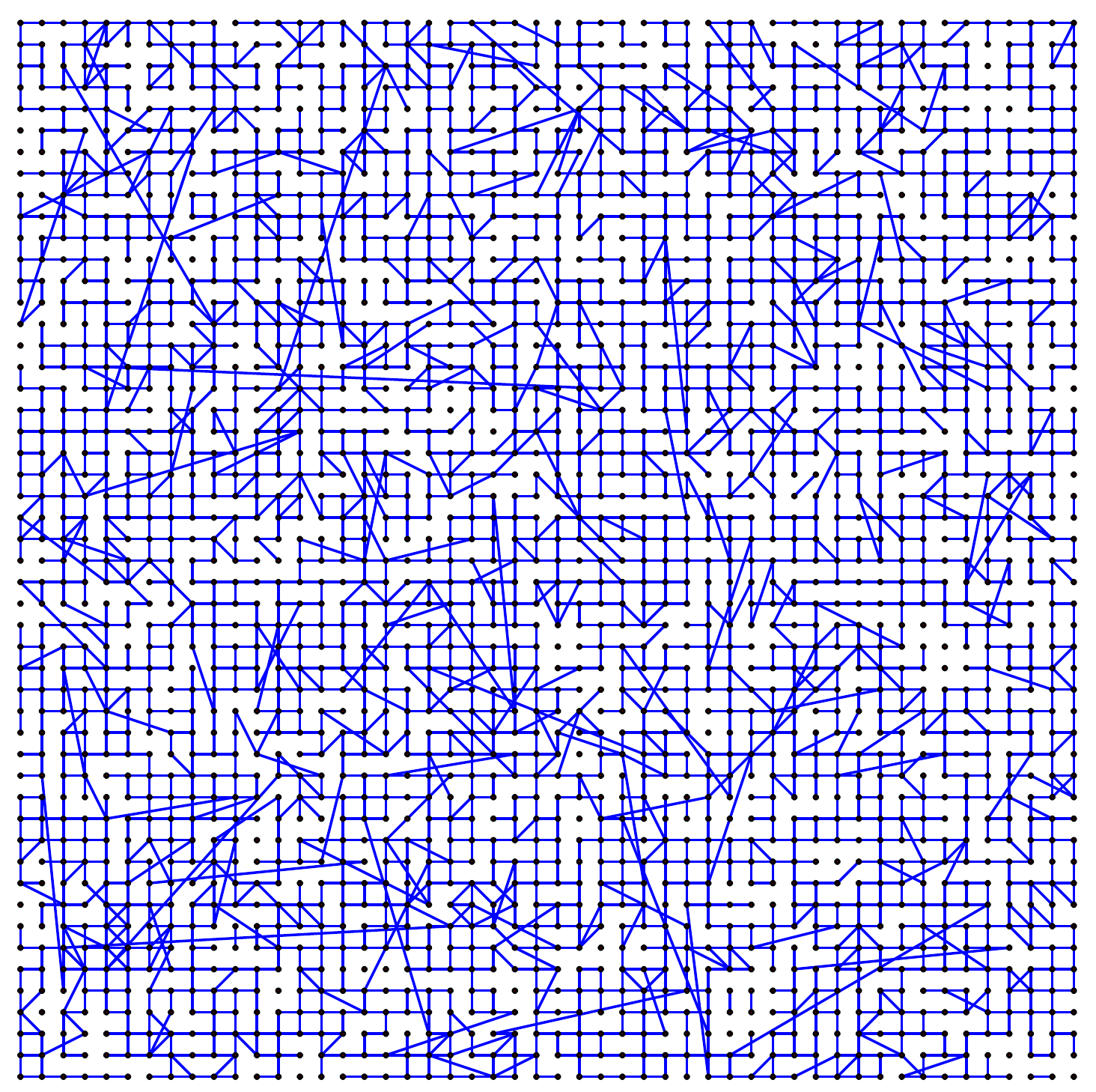}
		\caption{Long-range percolation for $\alpha=3.9, \lambda = 0.9$}\label{FigureLRP}
		\end{subfigure}	
		\begin{subfigure}{.96\linewidth}
		\includegraphics[keepaspectratio = True,width = 0.99\linewidth]{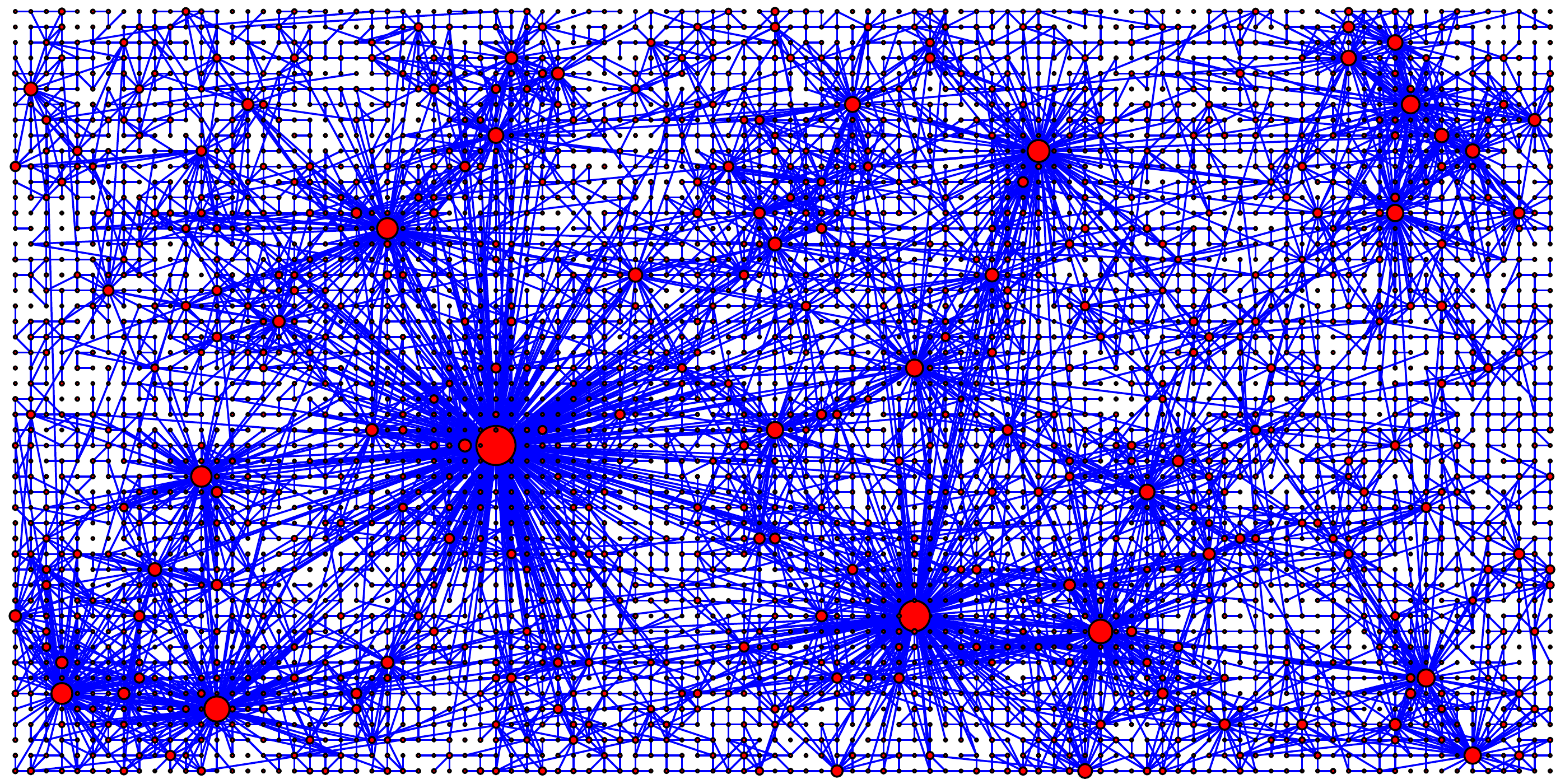}		
		\caption{Scale-free percolation for $\alpha = 3.9, \tau=1.95, \lambda = 0.1$}\label{FigureScalefree}
		\end{subfigure}
		\caption{Simulations of the Norros-Reittu random graph (A), long-range percolation (B), and scale-free percolation. The size of the vertices is drawn proportionally to their weights.}		
		\end{figure}

	Before we proceed with our results, let us briefly summarize some important features of SFP, as proved by Deijfen, van der Hofstad, and Hooghiemstra \cite{DeijfenScaleFree}, and by Deprez, Hazra, and W\"uthrich \cite{DeprezInhomogeneous}.	
		
It turns out that the following parameter is frequently useful to describe the behaviour of SFP concisely:
\begin{equation}
	\gamma:=\frac{\alpha(\tau-1)}{d}.
\end{equation}	

Like long-range percolation, SFP on $\Zd$ with parameter $\alpha$ and i.i.d.\ vertex weights whose law $W$ is $(\tau-1)$-regularly varying has a percolation phase-transition in $\lambda$ at 
\begin{equation}
	\lambda_c = \lambda_c(d,\alpha, W)
	:=\inf\big\{\lambda>0\,\big|\,\text{there exists an infinite cluster }\mathcal C_\infty\big\}.
\end{equation} 
This phase transition is non-trivial, except when $d\ge1$ and $\gamma<2$, in which case $\lambda_c=0$, and when $d=1$, $\gamma>2$, and $\alpha>2d$, in which case $\lambda_c=\infty$ \cite{DeijfenScaleFree}.
In the regime where SFP percolates, the infinite cluster $\mathcal{C}_\infty$ is almost surely unique \cite{Gandolfi}. Deprez \emph{et al.}\ show that the percolation density of SFP is continuous when $\alpha\in(d,2d)$: at $\lambda=\lambda_c$ there is no infinite cluster almost surely~\cite{DeprezInhomogeneous}.

		  By the choice of the power-law distribution, this model is scale-free. Indeed, the degrees $D$ follow a power-law of the form 
\[	\Prob(D>s)=s^{-\gamma}\ell(s)\]
for some slowly varying function $\ell(s)$ \cite{DeijfenScaleFree}.	  
		  This shows that the model behaves differently from long-range percolation. Many real-world networks are believed to have infinite variance degree distributions.  SFP has infinite variance degrees when $\gamma<2$. When $\gamma < 2$, SFP locally behaves like an ultra-small world \cite{DeijfenScaleFree}. 

	Under the assumption that the weights are bounded away from 0, the probability that an edge is open in scale-free percolation with parameters $\alpha, \tau$ and $\lambda$ stochastically dominates the probability that an edge is open in long-range percolation with parameters $\alpha$ and some $\lambda'>0$. Deprez \emph{et al.}\ \cite{DeprezInhomogeneous} use this domination to show that SFP locally has the small-world and clustering properties when $\alpha\in(d,2d)$, analogous to long-range percolation \cite{biskup2004}.  

\section{Main results}

\subsection*{Distances within the infinite percolation cluster}		
Given a graph $G = (V,E)$, the \emph{graph distance on $G$} between any $x,y \in V$ is defined as
\[
	d_{G}(x,y) = \# \text{\emph{ edges in $E$ on a shortest path from $x$ to $y$,}}
\]
with the conventions that $d_G(x,x) =0$ and $d_G(x,y) =\infty$ if $x$ and $y$ are not in the same connected component of the graph. We define the \emph{diameter} of $G$ as the maximal distance between two vertices in $G=(V,E)$, i.e., $\mathrm{diam}\, G = \max_{x,y \in V} d_{G}(x,y)$.

The infinite random subgraph $\Ccal$ of $(\Zd, \Zd \times \Zd)$ corresponding to the infinite component of supercritical SFP thus naturally produces a random metric on $\Zd$. We write $d_\Ccal$ for this metric. We write $x\wedge y$ for the minimum of $x$ and $y$.
Our first result is the proof of a conjecture by Deijfen \emph{et al.}\ \cite{DeijfenScaleFree}. 
	 
	\begin{theorem}[Finite diameter in the infinite-degree cases]\label{thmGraphDiameter}
Consider SFP on $\Zd$ with $d \ge 1$, $\lambda > 0$, and with i.i.d.\ vertex weights whose law $W$ satisfies for some $\tau >1$ and some $c >0$,
\begin{equation}\label{lowerBoundWeights}
\Prob(W\geq w)\geq cw^{-(\tau-1)}\wedge 1,\qquad \text{ for all } w>0.
\end{equation}
Then $\mathrm{diam}\, \Ccal = 2$ almost surely when $\gamma \le 1$, and $\mathrm{diam}\, \Ccal \le \lceil d/(d-\alpha)\rceil$ almost surely when $\alpha < d$.
\end{theorem}
Note that \eqref{lowerBoundWeights} implies $\P(W < c^{1/(\tau-1)}) =0$, thus the weights are bounded away from $0$.
See Figure \ref{phasesDistances} for an overview of the graph distances in which we combine the results of the present paper and those of \cite{DeijfenScaleFree,DeprezInhomogeneous}. Theorem \ref{thmGraphDiameter} thus complements the characterization of distances. Our proof for the case $\alpha<d$ is based on the proof of a similar result for long-range percolation with $\alpha<d$ by Benjamini, Kesten, Peres, and Schramm \cite{Benjamini}. 

For the Norros-Reittu random graph a similar result to Theorem \ref{thmGraphDiameter} is  known: van den Esker \emph{et al.}\ \cite{Esker} prove that when the weights are distributed as an infinite-mean power-law, then the diameter of the graph is almost surely $2$ or $3$  (more precise results are obtained under extra conditions). 

\subsection*{Transience and recurrence}
Graph distances are one way of characterizing the geometry of a graph. Another way of doing this is by studying the behaviour of random walk on the graph.
The notions of \emph{transience} and \emph{recurrence} are particularly relevant:
	\begin{definition}[Random walk, transience and recurrence]
	A simple random walk on a locally finite graph $G=(V,E)$ is a sequence $(X_n)_{n=0}^\infty$ with $X_0 \in V$ where $X_{n+1}$ is chosen uniformly at random from the ``neighbours'' of $X_n$, i.e., $$X_{n+1} \in \{x \in V \, : \, \{x, X_n\} \in E\},$$ independently of $X_0,\dots,X_{n-1}$. A graph is called \emph{recurrent} if for every $X_0$ a random walk returns almost surely to its starting point $X_0$. A graph is called \emph{transient} if it is not recurrent. 
	\end{definition}
 We prove the following two theorems, the results of which are summarized in the phase diagram in Figure \ref{phasesRandomWalk}.	

	\begin{figure}[ht!!]
		\centering
		\begin{subfigure}{\linewidth}
		\centering
		\includegraphics[keepaspectratio,height = .4\textheight]{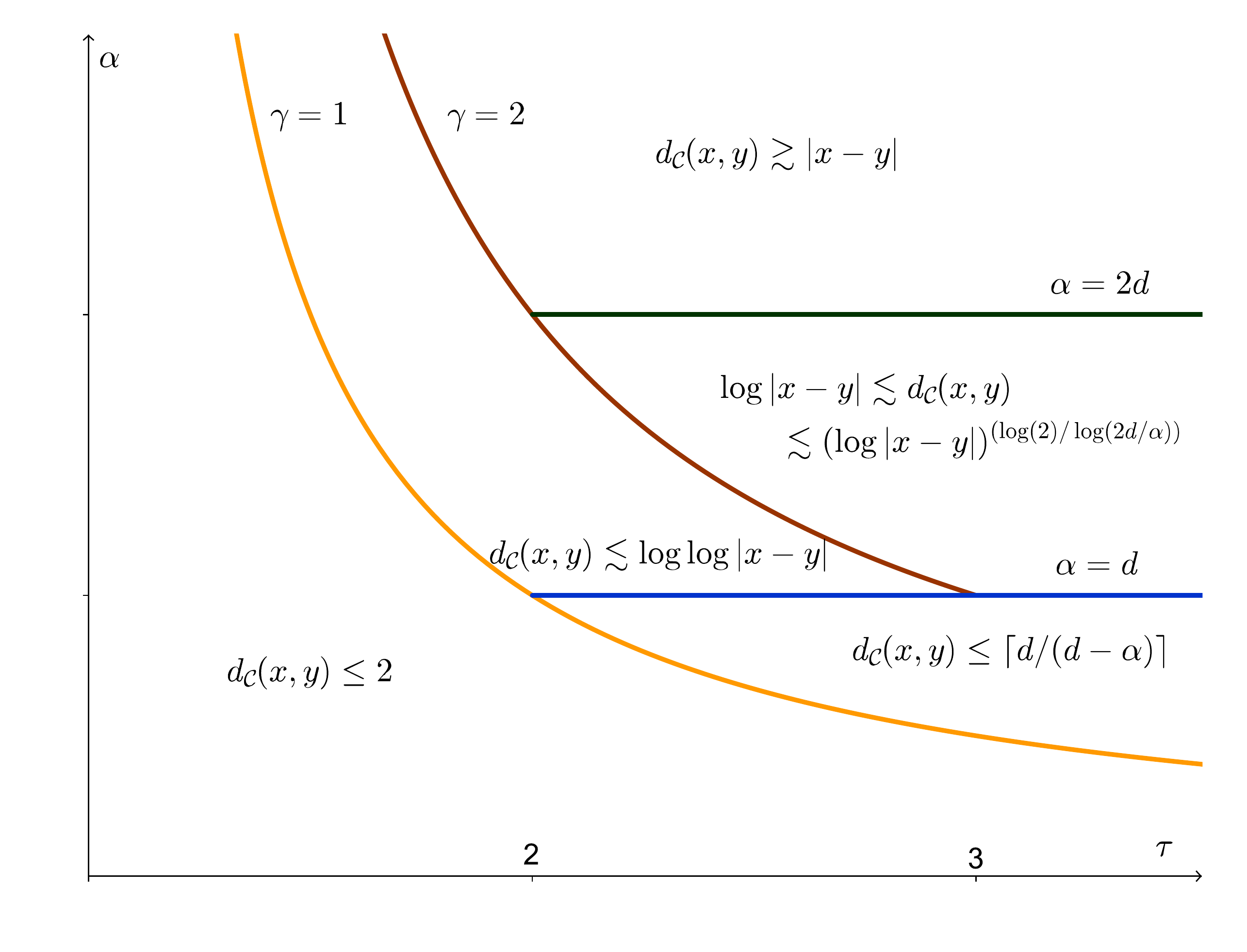}
		\caption{Overview of graph distances, combined results of Theorem \ref{thmGraphDiameter}, \cite{DeijfenScaleFree} and \cite{DeprezInhomogeneous}. 
		By the notation $d_{\Ccal}(x,y)\lesssim f(x,y)$ we mean that there exists a constant $c>0$, such that ${\lim_{|x-y|\rightarrow\infty}\Prob(d_{\Ccal}(x,y)\leq c f(x,y))=1}$. For $\gamma\in (1,2)$ and $\alpha>d$ stronger bounds have been proved~\cite[Theorem 5.1, 5.3]{DeijfenScaleFree}.
		 }\label{phasesDistances}
		\end{subfigure}
		\begin{subfigure}{\linewidth}
		\centering
		\includegraphics[keepaspectratio,height = .4\textheight]{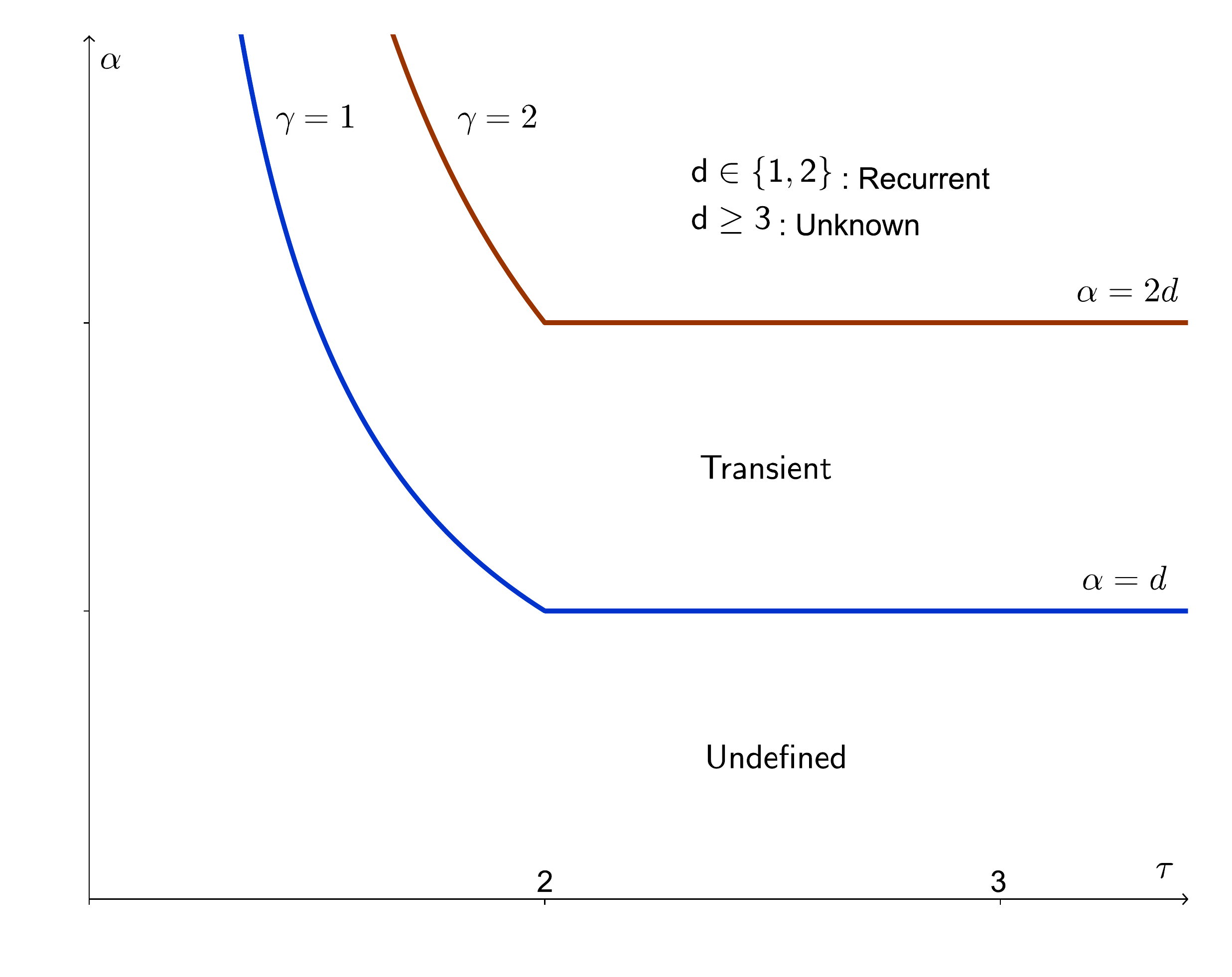}
		\caption{Recurrent vs.\ transient. Results of Theorem \ref{transientRandomThm} and \ref{recurrentRandomWalk2D}}\label{phasesRandomWalk}
		\end{subfigure}
		\caption{Phase diagrams. Transitions in $\gamma$ and $\alpha$. }
	\end{figure}
	\begin{theorem}[Transience in $d\geq 1$]\label{transientRandomThm} 
Consider SFP on $\Zd$ with $d \ge 1$, i.i.d.\ vertex weights whose law $W$ satisfies \eqref{lowerBoundWeights}, either $1 < \gamma < 2$ or $d < \alpha < 2d$, or both, and $\lambda > \lambda_c(d, \alpha, W)$.
Then the infinite cluster of SFP is transient almost surely.
\end{theorem}
Recall P\'olya's theorem, which states that the lattice of $\Zd$ with nearest neighbour edges is recurrent if and only if $d\in\{1,2\}$, and transient otherwise. Therefore, transience in these dimensions shows a dramatic difference to regular lattices. Berger \cite{BergerTransience} proved for LRP that the random walk is transient in one or two dimensions if and only if $\alpha\in(d,2d)$. For SFP the result is stronger: for any $\alpha>d$, there exists $\tau>1$ such that the infinite cluster is transient.
	\begin{theorem}[Recurrence in two dimensions]\label{recurrentRandomWalk2D} 
	Consider SFP on $\Zd$ with $d =2$, i.i.d.\ vertex weights whose law $W$ satisfies
\begin{equation}\label{upperBoundWeights}
	\Prob(W\geq w)\leq cw^{-(\tau-1)}, \qquad \text{for all }w\geq 0,
\end{equation}
for some $\tau >1$ and $c >0$, $\alpha > 4$, and $\lambda > \lambda_c(2, \alpha, W)$, and such that either $\tau >2$ or $\gamma >2$, or both. Then the infinite percolation cluster is recurrent $\Prob_{(\lambda, W)}$-almost surely.
\end{theorem}

	Note that, as mentioned before, in dimension 1 when $\gamma>2$ and $\alpha>2$ there is no infinite cluster almost surely \cite{DeijfenScaleFree}, so in this case a random walk is trivially recurrent. We therefore give a full characterization of recurrence and transience of SFP in dimension one and two, while for $d\geq3$ we only characterize it when $\alpha<2d$ or $\gamma<2$.
For nearest-neighbour percolation it is known that the infinite cluster is transient \cite{KestenRandomWalk}. 
It would be interesting  to verify whether this is true for other percolation models on $\Zd$, in particular for scale-free percolation or long-range percolation. 

\subsection*{Geometric clustering and hierarchies}	
	We show that SFP has the geometric clustering property not only for $\alpha\in(d,2d)$ as shown by Deprez \emph{et al.}\ \cite[Theorem 6]{DeprezInhomogeneous}, but also when $1< \gamma<2$. Moreover, these clusters can be organized in a hierarchical structure, a phenomenon that is also present in some real-life networks (see for example \cite[Chapter 13]{carrington2005models} or \cite[Chapter 9]{barabasi}). These hierarchical structures are not only present in finite boxes, they extend throughout $\Zd$. Indeed, the infinite component of SFP contains an infinite subgraph exhibiting a prescribed hierarchy. We introduce the notion of a \emph{hierarchically clustered tree}.
	\begin{definition}[Hierarchically clustered trees]\label{def:hct}
	Fix $m\geq 1$ and $x\in\Zd$. Let $\mathcal{Q}_m(x):=x+[0, m-1]^d\cap \Zd$. Consider the set of trees $\mathcal{T}_{x,m}$ of all unrooted, connected, cycle-free subgraphs of $(\mathcal{Q}_m(x), \mathcal{Q}_m(x) \times \mathcal{Q}_m(x))$ (i.e., trees on $\mathcal{Q}_m(x)$), where each vertex $v$ in such a tree is endowed with a weight $W_v \in \mathbb{R}$.
Fix  $\rho\in(0,1]$ and $K>0$. We call an element $T \in \mathcal{T}_{x,1}$ an \emph{$(x,1,K,\rho)$-hierarchically clustered tree} if $T = (\{x\}, \varnothing, \{W_x\})$ (i.e., $T$ is the isolated vertex $x$ with a weight). For $m \ge 2$, we call an element $T \in \mathcal{T}_{x,m}$ an \emph{$(x,m,\rho,K)$-hierarchically clustered tree} if
the following four properties hold:
\begin{enumerate}
\item\emph{[Positive density]} $T$ contains at least a fraction $\rho$ of all the vertices in the box $\mathcal{Q}_m(x)$:
\[
|V|>\rho m^d.
\]
\item\emph{[Ultra-small world]} $T$ is an ultra-small world in the sense that
\[
\mathrm{diam}\left(T\right)\leq K \;\max\{1,  \log\log m\}.
\]
\item\emph{[Ordered weights]} If we root $T$ at its maximum-weight vertex, then, for any vertex in the tree, the weights decrease step-by-step along the path from the root to that vertex.
\item\emph{[Spatial clustering]} If we remove any given edge from $T$, then there exists an $m'\le m$ (depending on $T$ and the removed edge) such that the two trees $T'_1=(V'_1, E'_1, W'_1)$ and $T'_2=(V'_2, E'_2, W'_2)$ that remain satisfy
\begin{enumerate}
\item at least one (say $T'_1$) is an $(x', m',\rho,K)$-hierarchically clustered tree for some $x'\in \mathcal{Q}_m(x)$, and 
\item the other (say $T'_2$) has its vertex set $V'_2$ disjoint with the box on which $T'_1$ is defined:
\[
\mathcal{Q}_{m'}(x')\cap V'_2 = \varnothing.
\]
\end{enumerate}
\end{enumerate}	
\end{definition}	
Note that condition $(1)$ together with condition $(2)$ implies that there exists $K'>0$, such that for all $m\geq 1$
\[
\mathrm{diam}\left(T\right) <K'\log\log\left|V_{T}\right|,
\]
so hierarchically clustered trees combine a topological and a spatial version of the ultra-small world property. 	
	\begin{theorem}[Hierarchically clustered trees]\label{thm:trees}
Consider SFP on $\Zd$ with $d \ge 1$, i.i.d.\ vertex weights whose law $W$ satisfies \eqref{lowerBoundWeights}, with $1<\gamma<2$, and any $\lambda >0$. 	
Let $\mathcal{S}_m$ denote the SFP configuration \emph{inside} the cube $[0,m-1]^d$. 
There exist $\xi >0$, a density $0 < \rho \le 1, K>0$ and a constant $m_0>0$, such that 
\begin{enumerate}
\item for all $m\geq m_0$,
\[
 	\Prob(\mathcal{S}_m\text{ contains a }(0, m, \rho,K)\text{-hierarchically clustered tree})\geq 1-\exp(-\rho m^\xi), \text{ and}
\]
\item the infinite component $\mathcal{C}_\infty$ contains a.s.\ an infinite, connected, cycle-free subgraph $\mathcal{T}_\infty$ such that if we remove any given edge from $\mathcal{T}_\infty$, a finite and infinite connected component remain and there exist $x\in\Zd$ and $m\geq 1$ such that the finite connected component is an $(x, m, \rho,K)$-hierarchically clustered tree.
\end{enumerate}
\end{theorem}

	\begin{figure}[t!]
	\centering	
		\includegraphics[width = 0.8\linewidth, height=0.2\textheight]{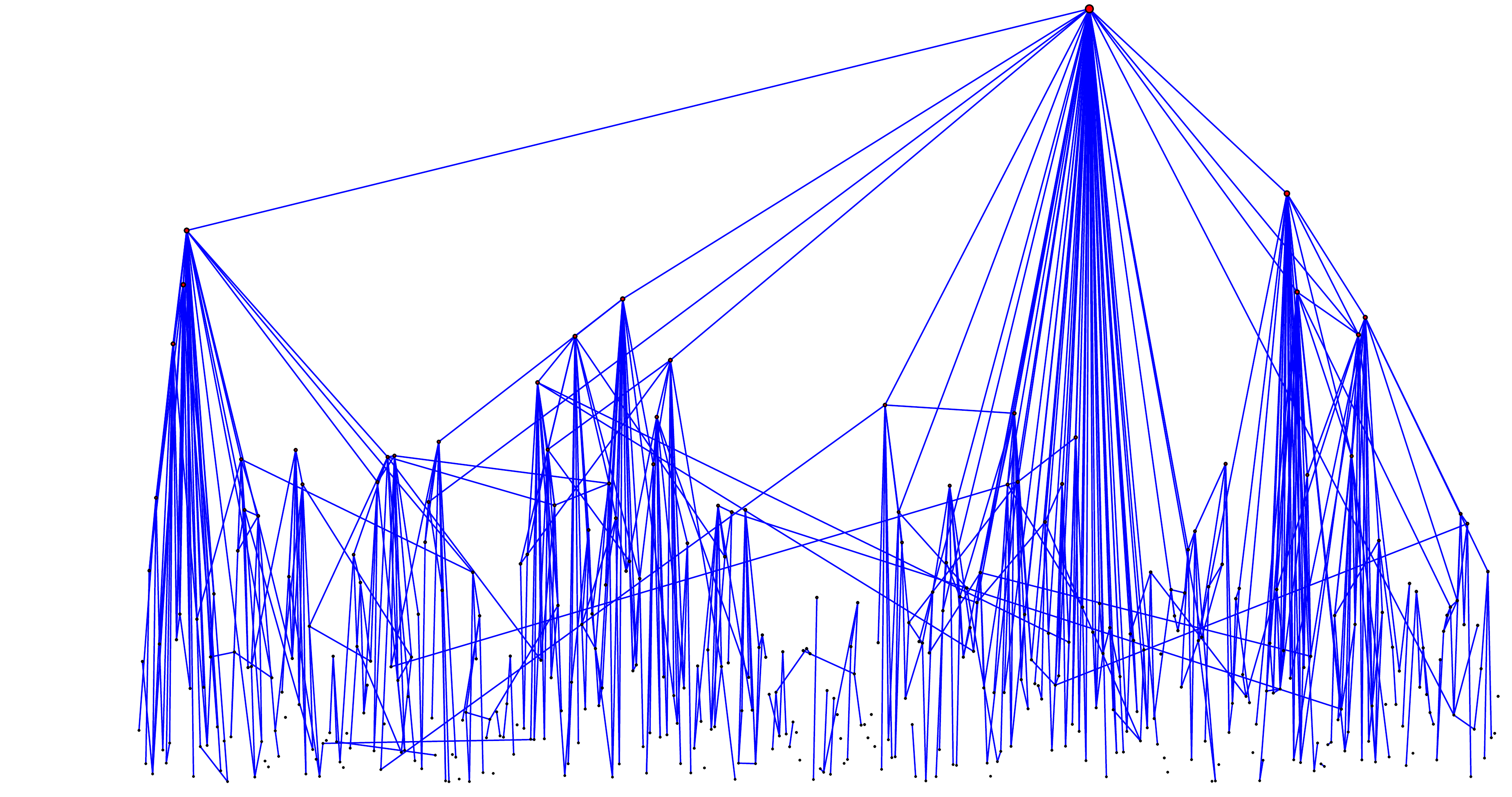}
		\caption{A simulation of scale-free percolation in $d=1$. The vertex-height in the figure depends on the weight (logarithmically). \mbox{$\alpha=2, \tau=1.95, \lambda = 0.1$}. }\label{FigureHierarchical}
	\end{figure}

\subsection*{Related results and open questions}\label{SectionOpenQuestions}
	\subsubsection*{Graph distance}
		This paper combined with \cite{DeijfenScaleFree, DeprezInhomogeneous} gives bounds on the graph distance for every value in the parameter space, but the picture is not yet complete. We were not able to prove a non-trivial lower bound on the diameter in the regime where $\gamma<2$ and $d > \alpha$, and it is not clear to us that the upper bound is sharp. And in the regime $\alpha\in(d,2d)$ where $\gamma>2$, there is a gap in the bounds on graph distances, since there the best known bounds are \cite{DeijfenScaleFree,DeprezInhomogeneous}:
		\[
		\lim_{|x|\rightarrow \infty} \Prob\left(c\log|x|\leq d_\mathcal{C}(0,x)\leq c^{-1}(\log|x|)^{\log(2)/\log(2d/\alpha)}\right)=1, \qquad \text{ for some } c>0.
		\] 		
		What is the right asymptotics of $d_\mathcal{C}(0,x)$ in this regime?

	\subsubsection*{Hierarchical structure}
	In Section \ref{SectionClusters} below we determine that the bound on $\xi$ in Theorem \ref{thm:trees} is ${\xi<\min \{d(2-\gamma)/(\tau+1),\tfrac{d}{2}(\tau+2-\sqrt{(\tau+2)^2-4(2-\gamma)})\}}$.
		Biskup \cite{biskup2004} shows a result rather similar to Theorem \ref{thm:trees} on the clustering density for long-range percolation when $\alpha\in(d,2d)$, where $\xi <d(2-\alpha)$. The corresponding range for $\xi$ for scale-free percolation would be $\xi <d(2-\gamma)$. It might be possible to extend Theorem \ref{thm:trees} to hold for this regime of $\xi$.

	\subsubsection*{Scale-free percolation on the torus.}	
	Scale-free percolation is defined as a model on the infinite lattice $\Zd$. 
	A challenging question is the study of scale-free percolation, and in particular its critical behaviour, on the finite torus. Working on the torus keeps the translation invariance and provides the opportunity to compare the model to its non-spatial counterparts, such as the Norros-Reittu random graph \cite{norros2006}. 
	
	Scale free percolation on finite boxes is strongly related to geometric variants of the Norros-Reittu model or the Chung-Lu model. 
	For example, Bringmann, Keusch, and Lengler \cite{BringmannGIRG} introduce \emph{geometric inhomogenous random graphs}, which generalise a certain class of hyperbolic random graphs, and which could be described as ``continuous SFP on the torus''. Indeed, they do not use a grid, but place the points randomly. In a fairly general setup, where, contrary to our model, the connection probability does not need to approach to 1 as $W_xW_y/|x-y|^\alpha$ goes to infinity, these authors prove that such graphs are ultra-small \cite{BringmannDistances}. Moreover, they claim that their results also carry over to finite boxes. Because of this more general setup, it would be interesting (but possibly not straightforward) to see whether in their setting hierarchically clustered trees are also present.

\subsection*{Organization}
The proofs of the main results partly rely on a number of elementary properties of the vertex weights. We begin by proving these properties in Section~\ref{SectionPreliminary}. In Section~\ref{SectionDistance} we prove the boundedness of the graph distance for $\alpha< d$ and $\gamma\leq 1$. In Section~\ref{SectionRW} we prove the random walk results, and in Section~\ref{SectionClusters} we prove Theorem \ref{thm:trees} 
on hierarchical clustering. 

\section{Preliminaries: properties of the vertex weights}\label{SectionPreliminary}
We start by introducing some basic notation and definitions. 
Given two percolation configurations $\omega, \omega' \in \{0,1\}^{\Zd \times \Zd}$, we write $\omega' \succcurlyeq \omega$ if $\omega'(e) =1$ when $\omega(e)=1$ for all $e \in \Zd \times \Zd$, i.e., all edges that are open in $\omega$ are also open in $\omega'$. We say that an event $A$ is \emph{increasing} if $\omega \in A$ implies $\omega' \in A$ for all $\omega' \succcurlyeq \omega$.

Given two random variables $X$ and $Y$, we say that $Y$ \emph{stochastically dominates} $X$ if for every $x \in \mathbb{R}$ the inequality $\P(X > x) \le \P(Y >x)$ holds, and we write $X \preceq_d Y$.

\begin{lemma}[Stochastic domination for SFP]\label{obs:increasing} Let $W$ and $W'$ be random variables such that $W' \preceq_d W$. For any increasing event $A$,
\begin{equation}\label{e:coupling}
	\P_{(\lambda,W)}(A) \ge \P_{(\lambda,W')}(A).
\end{equation}
\end{lemma}

This lemma can be proved with a straightforward coupling argument that we leave to the reader.
\medskip
	
We  commonly use Lemma \ref{obs:increasing} to simplify the law of $W$: If the law of $W$ satisfies \eqref{lowerBoundWeights} and the law of $W'$ satisfies
\begin{equation}\label{e:Wprime}
	\P(W' \ge w) = c w^{-(\tau-1)}, \qquad \text{ for all } w \ge c^{1/(\tau-1)},
\end{equation}
with the same constant $c$ as in \eqref{lowerBoundWeights}, then \eqref{e:coupling} holds.

The upcoming lemmas  allow us to construct a coarse-graining argument in the proofs of Theorems \ref{transientRandomThm} and \ref{thm:trees}. 
\begin{lemma}\label{powerLawMultiplyScalar}
Let $W$ be a random variable with law given by \eqref{e:Wprime}. Let $W''$ be a random variable with law given by
\begin{align}\label{standardPowerLaw}
\Prob(W''\geq w)&=w^{-(\tau-1)},\qquad \text{ for all }w\geq 1.
\end{align}
Then, for $y\geq c^{1/(\tau-1)}$, the conditional law of $W$ given $\{W\geq y\}$
is the same as
the law of $yW''$, i.e., $\Prob(W\geq x \mid W\geq y) = \P(y W'' \ge x)$.
\end{lemma}
\proof
For $x\geq y$
\[
\Prob(W\geq x \mid W\geq y) = \left(\frac{y}{x}\right)^{\tau-1}=\Prob\left(W''\geq\frac{x}{y}\right)=\Prob(yW'' \ge x).\qed
\]
\medskip

\begin{lemma}\label{maximumWeight}
Let $\{W_i\}_{i=1}^\infty$ be an i.i.d.\ sequence of random variables with law given by \eqref{e:Wprime}.  Then, for all $n\ge1$ and all $K_2\geq K_1\geq c^{1/(\tau-1)}$,
\[
\Prob\left(\max_{i=1,...,n}W_{i}\leq K_2 \, \Big| \, W_{i}\geq K_1 \text{ for }i=1,\dots ,n \right)\leq \exp\left(-n\left(\frac{K_1}{K_2}\right)^{\tau-1}\right).
\]
\end{lemma}
\medskip

\proof
Using that the weights are i.i.d., that $K_2\geq K_1$, and that $1-x\leq \exp(-x)$, we can bound the left-hand side by
\[
\left(1-\left(\frac{K_2}{K_1}\right)^{-(\tau-1)}\right)^n
\leq  \exp\left(-n\left(\frac{K_1}{K_2}\right)^{\tau-1}\right). \qed
\]
\medskip
\begin{lemma}\label{transientLemmaBigDegrees}
Fix an integer $d\ge1$ and $\alpha \in (0,\infty)$ such that $\gamma=\alpha(\tau-1)/d<2$. Assign to each vertex in $[0,N-1]^d \subset \Zd$ an i.i.d.\ random variable with law satisfying \eqref{lowerBoundWeights}. Let $E_{N,\beta}$ be the event that the box $[0,N-1]^d$ contains at least $\log  N$ vertices with weight larger than $\beta N^{\alpha/2}$. Then, for all $\beta>0$,
\[
\Prob\left(E_{N,\beta}\right)\longrightarrow 1,\qquad \text{ as }N\rightarrow\infty.
\]
\end{lemma}

\proof
Let $Y$ denote the number of vertices in $[0,N-1]^d$ with weight exceeding $\log N$. By \eqref{lowerBoundWeights} and independence of the weights we have $Y \preceq_d X$, where
$X\sim\text{Bin}(N^d,c\left(\beta N^{\alpha/2}\right)^{-(\tau-1)})$. 
Note that since $\gamma < 2$ we have $\E[X] = c \beta^{-(\tau-1)} N^{d(1-\gamma/2)} \gg \log N$ and Var$(X) \ll \E[X]^2$. It follows by the Paley-Zygmund inequality that (when $N$ is sufficiently large),
\[
	\P(E_{N, \beta}) \ge  \P(X \ge \log N) \ge \frac{(\E[X] - \log N)^2}{\mathrm{Var}(X) + \E[X]^2} 
	\longrightarrow 1. \qed
\]
\medskip

We call any set that is a translate of $[0,N-1]^d \subset \Zd$ an \emph{$N$-box.} We say that two $N$-boxes $\Qcal_1=v_1+[0,N-1]^d$ and $\Qcal_2=v_2+[0,N-1]^d$ are ``$k$ boxes away from each other'' if $\vert v_1-v_2\vert=kN$ (where we recall that $|\cdot|$ denotes the $\ell_1$-norm). 

\begin{lemma}\label{maxDistance} Let $d \ge 1$ and $k \ge 1$.
Consider two $N$-boxes $\Qcal_1$ and $\Qcal_2$ that are $k$ boxes away from each other. For arbitrary $u_1\in \Qcal_1$ and $u_2\in \Qcal_2$,
\[
\vert u_1-u_2\vert\leq 3dkN.
\]
\end{lemma}
\proof
Let $v_1$ and $v_2$ be such that $\Qcal_1=v_1+[0,N-1]^d$ and $\Qcal_2=v_2+[0,N-1]^d$.
Applying the triangle inequality twice, one obtains
\[
\vert u_1 -u_2\vert 
\leq \vert v_1 -v_2\vert+\vert u_1 -v_1\vert + \vert v_2 -u_2\vert 
\leq kN+2dN
\leq 3dkN.\qed
\]
\medskip

\begin{lemma}\label{transientLemmaConnectivity}
Fix $N \in \mathbb{N}$ and let $\mathcal{Q}_1$ and $\mathcal{Q}_2$ be two $N$-boxes that are $k$ boxes away from each other such that $\mathcal{Q}_1 = Nv_1 + [0,N]^d$ and $\mathcal{Q}_2 = Nv_2 + [0,N]^d$ with $v_1, v_2 \in \Zd$. Let $\beta>0$ be given, the weights $\{W_x\}_{x \in \Zd}$ be i.i.d.\ according to a law satisfying \eqref{e:Wprime}, and $\{W'_x\}_{x \in \Zd}$ be i.i.d.\ with law \eqref{standardPowerLaw}. 
For $i=1,2$ write
\[
	u_i = \argmax_{u \in \Qcal_i} W_u.
\]
Then
\[
\Prob_{(\lambda,W)}\left(\{u_1,u_2\}\text{ is open}\,\big|\,W_{u_1},W_{u_2}\geq \beta N^{\alpha/2}\right)
\geq 
\Prob_{(\lambda\beta^2(3d)^{-\alpha},W')}\Big(\{v_1,v_2\}\text{ is open}\Big).
\]
\end{lemma}
\proof
Let $U \sim $ Unif$[0,1]$ denote a standard uniform random variable with cdf $\P(U < x) =x$ for $x \in [0,1]$. Then, by Definition \ref{SFPDef},
\begin{multline*}
	\Prob_{(\lambda,W)}\left(\{u_1,u_2\}\text{ is open}\,\big|\,W_{u_1},W_{u_2}\geq \beta N^{\alpha/2}\right)\\
	 = \P^*\left(U < 1-\exp\left(-\lambda \frac{W_1 W_2}{|u_1 -u_2|^\alpha}\right) \mid W_1, W_2 \ge \beta N^{\alpha/2}\right),
\end{multline*}
where the probability measure $\P^*$ on the right-hand side is with respect to $W_1$ and $W_2$, which are i.i.d.\ with the same law as the elements of $\{W_x\}_{x \in \Zd}$, and an independent random variable $U \sim$ Unif$[0,1]$.
Using Lemmas \ref{powerLawMultiplyScalar} and \ref{maxDistance} we bound the right-hand side from below by
\[
	\P^{**}\left(U < 1-\exp\left(-\lambda \beta^2 (3d)^{-\alpha} \frac{W'_1 W'_2}{k^\alpha}\right)\right),
\]
where the probability measure $\P^{**}$ is with respect to $W'_1$ and $W'_2$ which are i.i.d.\ with the same law as the elements of $\{W'_x\}_{x \in \Zd}$ and an independent random variable $U \sim$ Unif$[0,1]$.

On the other hand, since $|v_1 -v_2|=k$, by Definition \ref{SFPDef} we also have
\[
	\P_{(\lambda \beta^2 (3d)^{-\alpha},W')}(\{v_1,v_2\} \text{ is open})
	 =  \P^{**}\left(U < 1-\exp\left(-\lambda \beta^2 (3d)^{-\alpha} \frac{W'_1 W'_2}{k^\alpha}\right)\right).
\]
The claim thus follows. \qed

\medskip

\section{Distances in the infinite degree case: proof of Theorem \ref{thmGraphDiameter}}\label{SectionDistance}
\proof[Proof of Theorem \ref{thmGraphDiameter}(1)]
[\emph{The case $\gamma\leq 1$}] By translation invariance of the model, it suffices to show  $\P(d_{\Ccal}(0,x)\leq 2)=1$.
Since we assumed that the law of $W$ satisfies \eqref{lowerBoundWeights}, there exists a $c>0$ such that $W_x \ge c^{1/(\tau-1)}$ for all $x \in \Zd$ almost surely.

Fix $x\in\Zd$. For $k\geq1$, let $\mathcal{Q}_k$ denote the box centred at $x/2$ with sides of length $l_k:=2^k|x|$ and let $\mathcal{A}_k :=\mathcal{Q}_k \setminus \mathcal{Q}_{k-1}$ with $\mathcal{A}_1 :=\mathcal{Q}_1$.  Note that there are $(2^d-1)|x|^d2^{d(k-1)}$ vertices in $\mathcal{A}_k$. 
\begin{figure}[hbt]
\centering
\includegraphics[keepaspectratio,width = 8cm]{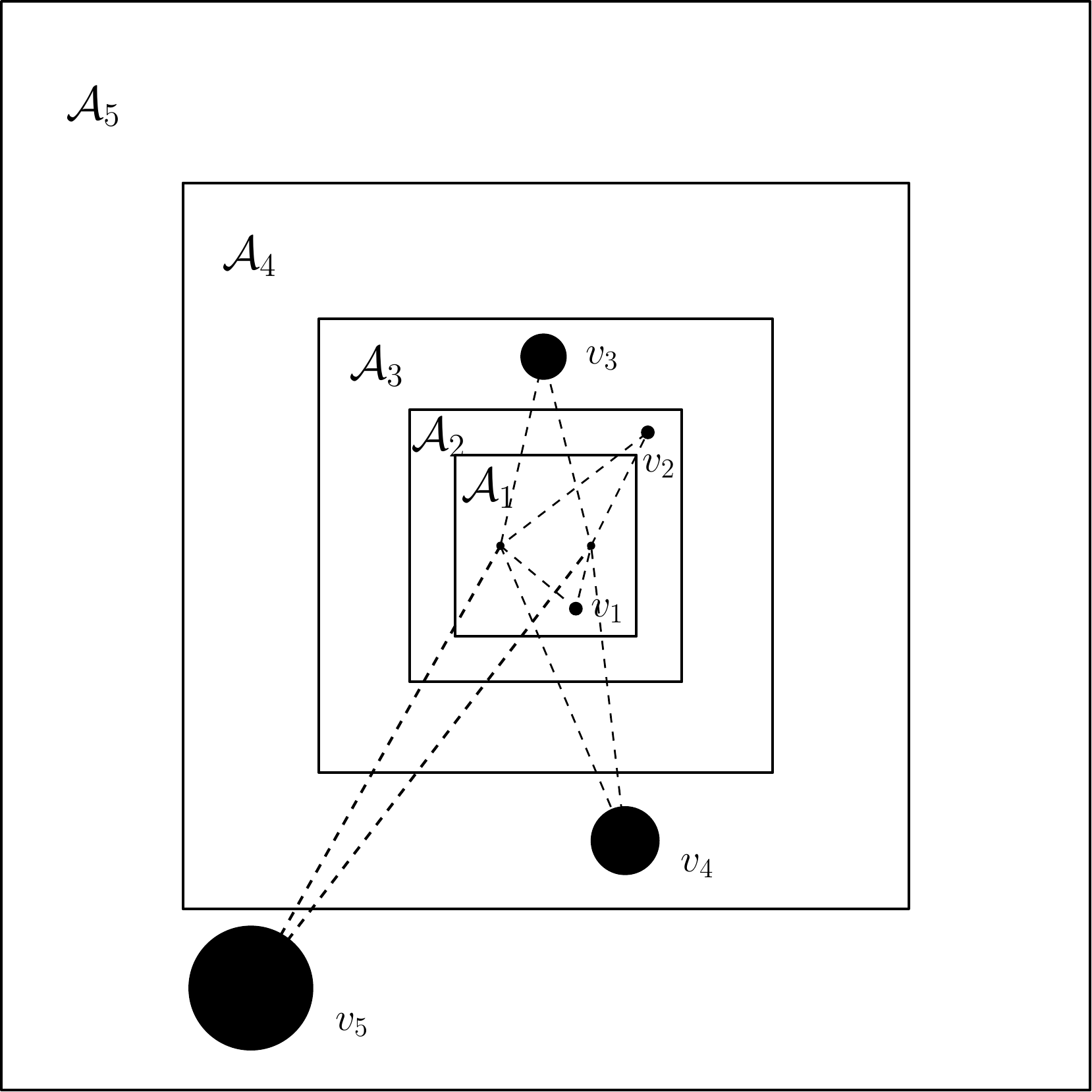}
\caption{Construction for the proof of Theorem \ref{thmGraphDiameter} for $d=2$.}\label{diameterFigure}
\end{figure}

We prove that the probability that the vertex with maximal weight for every $\mathcal{A}_k$ is connected to both 0 and $x$ is strictly greater than some positive constant and let the result follow by Borel-Cantelli. 
 
For each $k \in \mathbb{N}$, let $v_k$ be the vertex in $\mathcal{A}_k$ with maximal weight and let 
$E_k$ be the event 
 that $v_k$ is connected by an open edge to both 0 and $x$. Let $a_k := 2^{\frac{dk}{\tau-1}}$ and denote
$F_k :=\{W_{v_k}\geq a_k\}$, which is an increasing event.

Using Lemma \ref{obs:increasing} and Lemma \ref{maximumWeight} with $K_1=c^{1/(\tau-1)}$ and $K_2=2^{\frac{dk}{\tau-1}}$, we can bound
\begin{equation}\label{Fk-bound}
\begin{split}
\P_{(\lambda, W)}(F_k) & \ge \P_{(\lambda, W')}(F_k) 
\geq 1- \exp\left({-c(2^d-1)|x|^d2^{d(k-1)}2^{-\frac{dk}{\tau-1}(\tau-1)}}\right)\\
&=1-\exp\left(-\frac{c(2^d-1)|x|^d}{2}\right), 
\end{split}
\end{equation}
where the measure $\P_{(\lambda, W')}$ refers to a model where all weights are distributed as in \eqref{e:Wprime}. 
The right hand side of \eqref{Fk-bound} is bounded below by some $\delta > 0$ uniformly in $k$.

Observe that $|v_k|,|v_k-x|\leq dl_k$ and recall that $\tau>1$ and $\gamma\leq 1$. Write $\varepsilon = c^{1/(\tau-1)}$. We can bound the probabilities on the events $E_k$ by conditioning on $F_k$ as follows:
\begin{align*}
\Prob_{(\lambda, W)}(E_k \mid F_k) \ge \Prob_{(\lambda, W')}(E_k \mid F_k) \ge \Prob_{(\lambda, a_k)}(E_k)
&\geq \left(1-\exp\left(-\frac{\lambda \varepsilon a_k }{(dl_k)^\alpha}\right)\right)^2\\
&=\left(1-\exp\left(-\frac{\lambda \varepsilon}{(d|x|)^\alpha}2^{dk/(\tau-1)-k\alpha}\right)\right)^2\\
&\geq \frac{1}{4}\left(\left(\frac{\lambda \varepsilon}{(d|x|)^\alpha}2^{dk(1/(\tau-1)-\alpha/d)}\right)^2\wedge1\right)\\
&= \left(\left(\frac{\lambda \varepsilon}{2(d|x|)^{\alpha}}\right)^2\left(4^{d(1-\gamma)/(\tau-1)}\right)^k\right)\wedge\frac{1}{4} \\
&\ge \left(\left(\frac{\lambda \varepsilon}{2(d|x|)^{\alpha}}\right)^2\right)\wedge\frac{1}{4} =: \eta.
\end{align*}
Since this bound is independent of $k$ and of the weights $\{W_x\}_{x \in \Zd}$, it follows that $\P_{(\lambda,W)}(E_k \mid F_k) \ge \eta$, and therefore, 
\begin{equation*}
\P_{(\lambda,W)}(E_k) = \P_{(\lambda,W)}(E_k \mid F_k)\;\P_{(\lambda, W)} (F_k)\geq \eta \, \delta >0.
\end{equation*} 

Observe that the events $E_k$ are independent of each other, hence we obtain the result for $\gamma\leq 1$ using the Lemma of Borel-Cantelli. \qed
\medskip

\proof [Proof of Theorem \ref{thmGraphDiameter}(2)]
[\emph{The case $\alpha <d$}]
By translation invariance it again suffices to show that 
\[\P_{(\lambda,W)}(d_{\Ccal}(0,x) \le \lceil d /(d-\alpha)\rceil) =1\] 
for all $x \in \Zd$. Recall that the assumption \eqref{lowerBoundWeights} on the law of $W$ implies that $W \ge c^{1/(\tau-1)}$ almost surely. Note that $\{ d_{\Ccal}(0,x) \le \lceil d /(d-\alpha)\rceil \}$ is an increasing event. Hence by Lemma~\ref{obs:increasing},
\[
	\P_{(\lambda,W)}(d_{\Ccal}(0,x) \le \lceil d /(d-\alpha)\rceil) \ge \P_{(\lambda,c^{1/(\tau-1)})}(d_{\Ccal}(0,x) \le \lceil d /(d-\alpha)\rceil).
\]
Observe that SFP with constant vertex weights is equivalent to long-range percolation with the same $d$ and $\alpha$ and some possibly different parameter $\lambda'$. 

Benjamini \emph{et al.\ }\cite[Example 6.1]{Benjamini} show that the diameter of the infinite cluster in long-range percolation with $\alpha < d$ for any $\lambda>0$ is equal to $\lceil d/(d-\alpha)\rceil$ almost surely. Our claim about SFP therefore follows. \qed

\section{Transience vs.\ recurrence}\label{SectionRW}
\subsection*{Transience proof}
The proof of Theorem \ref{transientRandomThm} is inspired by Berger's proof of transience for long-range percolation \cite[Theorem 1.4(II)]{BergerTransience}. We use in particular a multiscale ansatz which roots back to the work of Newman and Schulman \cite{newman1986} for long-range percolation.
\medskip

\textbf{The case $1< \gamma<2$.} 
In view of Lemma \ref{obs:increasing}, we may assume  \eqref{e:Wprime} rather than \eqref{lowerBoundWeights} without loss of generality. 
We show that the infinite cluster of SFP almost surely contains a transient subgraph. The proof has two steps:
\begin{enumerate}
\item We first assume that $\lambda$ is large enough. With small probability we remove some vertices from the graph independent of each other. Then we use a multiscale ansatz: we group vertices into finite boxes, and call boxes `good' or `bad' according to the weights and edge structure \emph{inside} the box. We iterate this process by considering larger boxes, which we call good or bad according to the number of good boxes in them and the edges between vertices in those boxes. This will imply transience for large values of $\lambda$. 
\item To couple the original model (for \emph{any} $\lambda>\lambda_c$) to the model of the first step, we use a coarse-graining argument: We `zoom out' by considering large boxes of vertices and only considering the vertices with maximum weight in the boxes. We show that, with high probability, the weights of these vertices are so high, that the graph, only defined on these vertices, dominates a graph as described in the first step.
\end{enumerate}

We use \cite[Lemma 2.7]{BergerTransience}, which describes a sufficient structure for a graph to be transient. To this end, we introduce the notion of a ``renormalized graph'':

We start with some notation. Given a graph $G=(V,E)$ and a sequence $\{C_n\}_{n=1}^\infty$ let $V_l(j_l,\dots,j_1)$ with $l \in \mathbb{N}$ and $j_n \in \{1,\dots,C_n\}$ be a subset of the vertex set $V$. Now let for $l \ge m$
\[
	V_l(j_l,\dots,j_m) = \bigcup_{j_{m-1} =1}^{C_{m-1}} \dotsm \bigcup_{j_1=1 }^{C_1} V_l(j_l,\dots,j_1).
\]
We call the sets $V_l(j_l,\dots,j_m)$ \emph{bags}, and the numbers $C_n$ \emph{bag sizes.}

\begin{definition}
We say that the graph $G=(V,E)$ is \emph{renormalized for the sequence} $\{C_n\}_{n=1}^\infty$ if we can construct an infinite sequence of graphs such that the vertices of the \emph{$l$-th stage graph} are labelled by $V_{l}(j_l, \dots, j_1)$ for all $j_n \in \{1,\dots, C_n\}$, and such that for every $l \ge m >2$, every $j_l,\dots,j_{m+1}$, and all pairs of distinct $u_m, w_m \in \{1,\dots,C_m\}$ and $u_{m-1}, w_{m-1} \in \{1, \dots, C_{m-1}\}$ there is an edge in $G$ between a vertex in $V_l(j_l, \dots, j_{m+1}, u_m, u_{m-1})$ and a vertex in $V_l( j_l, \dots, j_{m+1}, w_m, w_{m-1}).$
\end{definition}
The underlying intuition is that every $n$-th stage bag contains $C_n$ $(n-1)$-stage bags, which contains again $C_{n-1}$ $(n-2)$-stage bags. Every pair of $(n-2)$-stage bags in an $n$ stage bag is connected by an edge between one of the vertices in the bags (see Figure \ref{renormalizedGraph}).
\begin{figure}
\centering
\includegraphics[keepaspectratio,width = \textwidth]{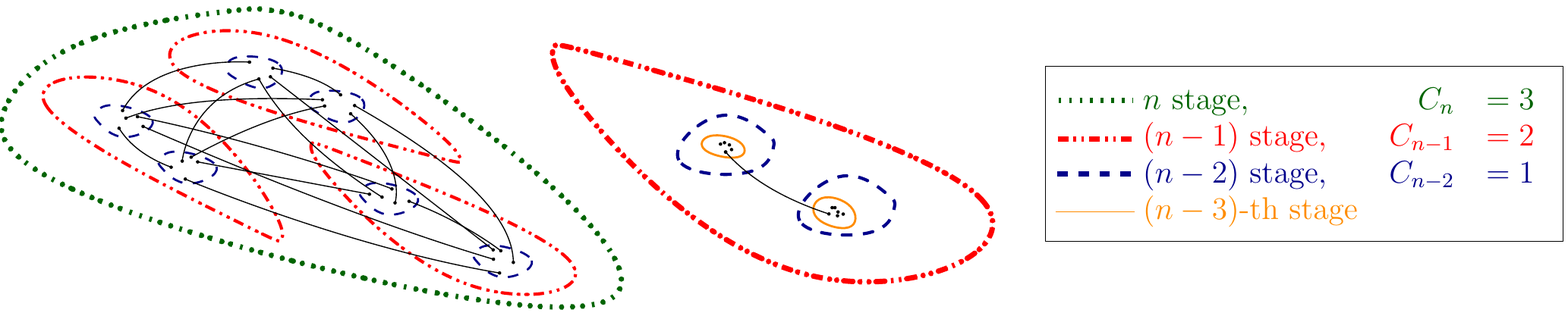}
\caption{$n$ and $(n-1)$ stage bag of a renormalized graph}\label{renormalizedGraph}
\end{figure}

\begin{lemma}[Berger, {\cite[Lemma 2.7]{BergerTransience}}]\label{renormalizedTransient}
A graph renormalized for the sequence $C_n$ is transient if $\sum_{n=1}^\infty C_n^{-1}<\infty$.
\end{lemma}
The lemma follows from the proof of \cite[Lemma 2.7]{BergerTransience}. 

\begin{proposition}\label{lemmaTransientLargeLambda}
Consider scale-free percolation with $\gamma<2$ and weight distribution satisfying \eqref{standardPowerLaw}.
Independently of this, perform an i.i.d.\ Bernoulli site percolation on the vertices of $\Zd$, colouring a vertex ``green'' with probability $\mu \in (0,1]$.

Then the subgraph of the infinite scale-free percolation cluster that is induced by the green vertices has a (unique) infinite component $\Ccal_{\lambda,\mu}$.
There exists $\mu_0<1$ and $\lambda_0 > 0$, such that $\Ccal_{\lambda,\mu}$ is transient for $\mu\geq \mu_0$ and $\lambda\geq\lambda_0$ almost surely.
\end{proposition}

The proof exploits a multiscale technique. Indeed, we proceed by showing that $\Ccal_{\lambda,\mu}$ contains a renormalized subgraph that is transient. Therefore, $\Ccal_{\lambda,\mu}$ is also transient. 

\proof[Proof of Proposition \ref{lemmaTransientLargeLambda}]
For all $n \in \mathbb{N}$, let 
\[
	D_n := 2(n+1)^2, \qquad C_n := (n+1)^{2d},
\]
and
\[
	u_n := d^{\alpha/2}(n+2)^{d(2-\gamma)/2}2^{(n+2)\alpha/2}((n+3)!)^\alpha.
\]
We partition the lattice $\Zd$ into disjoint boxes of side length $D_1$, so that each such box contains $D_1^d$ vertices, and call these the \emph{1-stage} boxes. (By convention we call vertices of $\Zd$ the \emph{0-stage} boxes.) We view these boxes as the vertices of a renormalized lattice. Now cover the lattice again, grouping together $(D_2)^d$ 1-stage boxes to form \emph{2-stage} boxes with sides of length $D_2$. Continue in this fashion, so that the \emph{$n$-stage} boxes form a covering of $\Zd$ by translates of $[0,\prod_{k=1}^n D_k-1]^d$.

We call a 0-stage box ``good'' if the vertex associated with it is green. 

For every stage $i \geq 1$, we define rules for a box to be ``good'' or ``bad'', depending only on the weights $W_x$ and the edges of $\Ccal$ inside the box. This implies that disjoint boxes are good or bad independently of each other.

A 1-stage box is good if it contains at least $C_1$ good 0-stage boxes and one of the vertices in these boxes has weight at least $u_1$. For each good $1$-stage box, call the maximum-weight vertex, having weight at least $u_1$, and call it \emph{1-dominant.}

For $n\geq 2,$ say that an $n$-stage box $\mathcal{Q}$ is good if the following three conditions are satisfied:
\begin{enumerate}
\item[(E)] At least $C_n$ of the $(n-1)$-stage boxes in $\mathcal{Q}$ are good. 
\item[(F)] For any good $(n-1)$-stage box $\Qcal' \subset \Qcal$, the $(n-2)$-dominant vertices in $\Qcal'$ form a \emph{clique} (i.e., every two $(n-2)$-dominant vertices in $\Qcal'$ are connected by an edge in $\Ccal$).
\item[(G)] There is an $(n-1)$-dominant vertex in one of its good $(n-1)$-stage boxes, with weight at least $u_{n}$.
\end{enumerate} 
For each good $n$-stage box, choose the maximum weight vertex  and call it the \emph{$n$-dominant} vertex if its weight is at least $u_{n}$. (A vertex may be dominant for different values of $n$.)
\begin{figure}
\centering
\includegraphics[width = 0.8\textwidth]{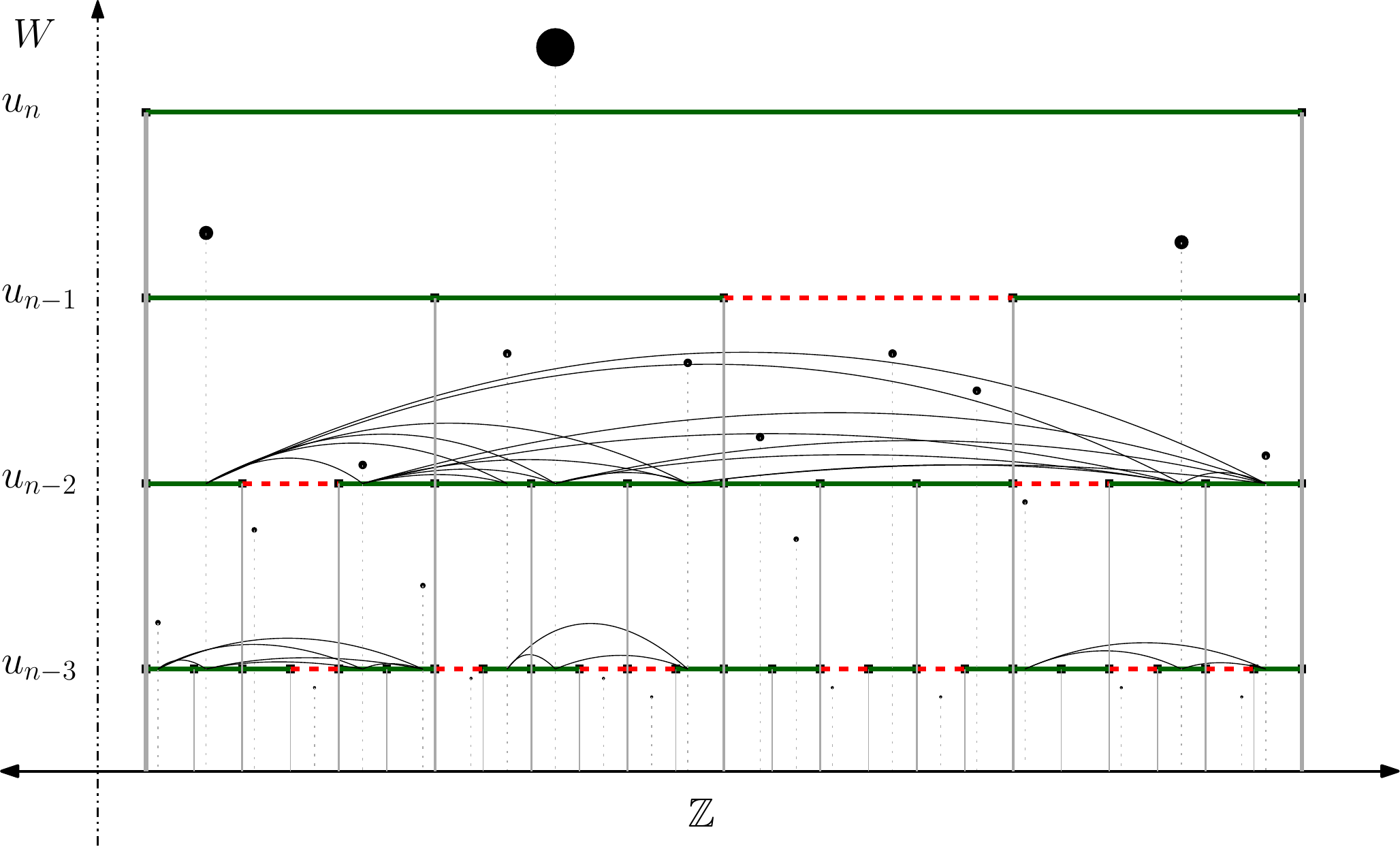}
\centering
\caption{Sketch of the renormalization in Theorem \ref{transientRandomThm} in $d=1$ for \\$D_n=4, D_{n-1}=3, D_{n-2}=2, C_n=3, C_{n-1}=2, C_{n-2}=1$. \\
`Good' boxes are marked with a solid line, `bad' boxes have a dashed line. 
}
\label{renormalizationTransient}
\end{figure}
See Figure \ref{renormalizationTransient} for a sketch of this definition.

Note that by construction, the subgraph of $\Ccal$ induced by the vertices that are in a good $n$-stage box for every $n \ge 0$ is a graph renormalized by a sequence of bag sizes $\{C_n\}$ that satisfies the transience condition of Lemma \ref{renormalizedTransient}. Our aim is therefore to show that almost surely such a subgraph exists.

 Define $E_n(v), F_n(v)$ and $G_n(v)$ to be the events that conditions (E), (F) and (G) hold for the $n$-stage box containing the vertex $v$. To simplify notation, define $E_n:=E_n(0),F_n:=F_n(0)$ and $G_n:=G_n(0)$. We write $L_n(v)$ and $L_n$ for the events that the $n$-th stage boxes containing $v$ and $0$, respectively, are good. By translation invariance it is sufficient to show that 
 \[\Prob\Bigg(\bigcap_{n=1}^\infty L_n\Bigg)>0.\] 
 The events $L_n$ are positively correlated, hence it is sufficient to show that \[\prod_{n=1}^\infty \Prob(L_n)>0.\] 

We bound
\begin{equation}\label{EqLnTrans}
\Prob(L_n^c)\leq\Prob(E_n^c)+\Prob(F_n^c \mid E_n)+\Prob(G_n^c \mid E_n).
\end{equation}
First, we give an upper bound for $\Prob(F_n^c \mid E_n)$. Recall that we use the $\ell_1$-norm for distance in the definition of the edge-probabilities of SFP. The $\ell_1$-distance between two vertices in the same $n$-stage box is at most
\[
d\prod_{k=1}^nD_k=d2^n((n+1)!)^2.
\]
The probability that two good $(n-2)$-stage boxes are \emph{not} connected by an open edge between its $(n-2)$-dominant vertices (which have weight at least $u_{n-2}$) is therefore at most
\begin{align*}
\exp\left(-\lambda d^{-\alpha}u_{n-2}^2\prod_{k=1}^nD_k^{-\alpha}\right)&=\exp\left(-\lambda d^{-\alpha}\left(d^{\alpha/2}n^{d(2-\gamma)/2}2^{n\alpha/2}[(n+1)!]^\alpha\right)^22^{-n\alpha}((n+1)!)^{-2\alpha}\right)\\
&=\exp(-\lambda n^{d(2-\gamma)}).
\end{align*}
There are 
\[
\binom{D^d_nD^d_{n-1}}{2}
<4^d(n+1)^{4d}
\]
pairs of $(n-2)$-stage boxes inside an $n$-stage box, so there can be at most $4^d(n+1)^{4d}$ edges between $(n-2)$-dominant vertices inside a good $n$-stage box.
It follows by taking the union bound that
\begin{equation}\label{upperFc}
\Prob(F_n^c \mid E_n)\leq\exp\left(d\log(4)+4d\log(n+1)-\lambda n^{d(2-\gamma)}\right).
\end{equation}
\medskip

We proceed by establishing an upper bound on $\Prob(G_n^c \mid E_n)$. There exists a constant $c_1>0$ such that
\begin{equation}\label{fractionWeights}
\begin{split}
\frac{u_{n-1}}{u_{n}}& =2^{-\alpha/2}\left(\frac{n+1}{n+2}\right)^{d(2-\gamma)/2}\frac{1}{(n+3)^\alpha}\\
& \geq c_1(n+1)^{-\alpha}.
\end{split}
\end{equation}

Note that any good $n$-stage box contains at least $C_{n}$ $(n-1)$-dominant vertices that all have weight larger than $u_{n-1}$. Using \eqref{fractionWeights}, Lemma \ref{maximumWeight}, and $\gamma=\alpha(\tau-1)/d$, gives for some $c_2>0$ that
\begin{equation}\label{upperGc}
\begin{split}
\Prob(G_n^c \mid E_n) & \leq \exp\left(-C_{n}\left(\frac{u_{n-1}}{u_{n}}\right)^{\tau-1}\right)\\
& \leq\exp\left(c_1^{\tau-1}(n+1)^{2d-\alpha(\tau-1)}\right)\\
&\leq \exp\left(-c_2n^{d(2-\gamma)}\right).
\end{split}
\end{equation}
\medskip

The last term we bound is $\Prob(E_n^c)$. All $(n-1)$-stage boxes are good independent of each other with probability $\Prob(L_{n-1})$. Let $X\sim \text{Bin}(D_n^d,\Prob(L_{n-1}))$ be binomially distributed, so that $\Prob(E_n^c)=\Prob(X<C_n)$.  
We use Chernoff's bound that if $X\sim\text{Bin}(m,p)$, $\theta \in(0,1)$, then  $\Prob(X<(1-\theta)mp)\leq\exp(-{\frac12 \theta^2mp}).$
For our model, this obtains 
\begin{equation}\label{upperEc}
\begin{split}
\Prob(E_n^c) & \leq\exp\left(-\frac{1}{2}\left(1-\frac{1}{2\Prob(L_{n-1})}\right)^2\Prob(L_{n-1}) D_n^d \right)\\
& \leq \exp\left(-2^{d-3}(2\Prob(L_{n-1})-1)^2 (n+1)^{2d}\right).
\end{split}
\end{equation}

Combining \eqref{EqLnTrans}, \eqref{upperFc}, \eqref{upperGc}, and \eqref{upperEc}, gives that 
\begin{equation*}
\begin{split}
	\Prob(L_n^c) & \leq  \exp(d\log(4)+4d\log(n+1)-\lambda n^{d(2-\gamma)}) + \exp\left(-c_2n^{d(2-\gamma)}\right)\\
	& \quad +\exp\left(-2^{d-3}(2\Prob(L_{n-1})-1)^2(n+1)^{2d}\right).
\end{split}
\end{equation*}

If $\lambda$ large enough (say larger than $\lambda_0$), there exists $n_0$ such that for $n\geq n_0$
\begin{equation}\label{upperEFG}
\Prob(L_n^c)\leq 2\exp\left(-c_2n^{d(2-\gamma)}\right)+\exp\left(-2^{d-3}(2\Prob(L_{n-1})-1)^2(n+1)^{2d}\right).
\end{equation}
Define the sequence \[\ell_n:=1-(n+1)^{-3/2}\] and observe that
\begin{equation}\label{prodPositive}
\prod_{n=1}^\infty \ell_n>0.
\end{equation}
For any fixed $n_1>n_0$, we can find $\lambda_0>0$ and $\mu_0<1$ such that 
 $\Prob(L_{n_1})\geq \ell_{n_1}$, because $L_{n_1}$ depends only on the weights and edges inside a \emph{finite} box.
We further bound \eqref{upperEFG} for all $n > n_1$ by 
\begin{align}\label{recursiveBound}
\Prob(L_n^c)&\leq\exp\left(-c_2n^{d(2-\gamma)}\right)+\exp\left(-2^{d-3}\left(1-\frac{1}{\sqrt{2}}\right)^2(n+1)^{2d}\right)\nonumber\\
&\leq (n+1)^{-3/2}
=1-\ell_n,
\end{align}
and choose $n_1$ so large, that the last bound in \eqref{recursiveBound} holds.  
Thus, using \eqref{prodPositive}, \eqref{recursiveBound} and $\Prob(L_n)>0$ for all $n$, yields that
\begin{equation*}
\prod_{n=1}^\infty\Prob(L_n)=\prod_{n=1}^{n_1}\Prob(L_n)\prod_{n=n_1+1}^\infty \Prob(L_n)\geq \prod_{n=1}^{n_1}\Prob(L_n)\prod_{n=n_1+1}^\infty \ell_n>0.
\end{equation*}
With probability 1 the graph contains a cluster of good vertices that can be renormalized for the sequence $C_n$. By Lemma \ref{renormalizedTransient} this cluster is transient itself, since showing transience for a subgraph is enough for transience on the whole graph \cite[Section 9]{LectureNotesPeres}. \qed
\medskip

\textbf{The case $d< \alpha<2d$.} 
We need two lemmas from the literature, which are complementary to the case $\alpha\in(d,2d)$ of Proposition \ref{lemmaTransientLargeLambda} and Lemma \ref{transientLemmaBigDegrees}.

\begin{lemma}[Deprez, Hazra \& W\"uthrich, {\cite[Lemma 9]{DeprezInhomogeneous} }]\label{clustersizeTransient}Assume $\gamma>1$ and let $\alpha\in(d,2d)$. Choose $\lambda>\lambda_c$ and let $\alpha'\in[\alpha,2d)$. For every $\mu\in[0,1)$ and $\beta>0$ there exists $M_0\geq 1$ such that for all $m\geq M_0$
\[
\Prob\left(|\mathcal{C}_m|\geq \beta m^{\alpha'/2}\right)\geq \mu,
\]
where $\mathcal{C}_m$ is the largest connected component in $[0,m-1]^d$.
\end{lemma}
Note that \cite[Lemma 9]{DeprezInhomogeneous} is proven for the exact power law distribution of the weights in \eqref{standardPowerLaw}. This is not a problem, since Lemma \ref{obs:increasing} implies that the result extends to a weight distribution satisfying \eqref{lowerBoundWeights} when $c\geq 1$. 
For $c<1$, and percolation parameter $\lambda'>0$, we observe that the model is equivalent to the case where $c=1$ and $\lambda=\lambda'c'^{-2/(\tau-1)}$, since if the law of $W$ satisfies \eqref{e:Wprime} for some $c>0$ and for $W'$ it holds for $w \ge 1$ that $\Prob(W'>w)=w^{-(\tau-1)}$, then $W\overset{d}{=}c^{-1/(\tau-1)}W'.$ Hence, we can scale the parameters such that $c=1$ and apply Lemma \ref{obs:increasing}.

\begin{lemma}[Berger, {\cite[Lemma 2.7]{BergerTransience}}]\label{sitebondlongrangeLemma}
Let $d\geq1$, $\alpha\in(d,2d)$ and $\lambda>0$. Consider the long-range percolation model on $\Zd$ in which every two vertices, $x$ and $y$, are connected by an open edge with probability \[1-\exp\left(-\frac{\lambda}{|x-y|^\alpha}\right),\] independently of other edges, and every vertex is good with probability $\mu<1$, independently of all other vertices. 

There exist $\mu_1<1$ and $\lambda_1$, such that if $\lambda\geq\lambda_1$ and $\mu\geq\mu_1$, the infinite cluster on the good vertices is transient.
\end{lemma}
\begin{lemma}\label{largeConnectivity}
Consider scale-free percolation with weight distribution satisfying \eqref{lowerBoundWeights}. 
Let $\mathcal{Q}_1$ and $\mathcal{Q}_2$ be two $N$-boxes that are $k$ boxes away from each other. Let $\beta>0$ be given. Moreover, assume that $\mathcal{Q}_1$ and $\mathcal{Q}_2$ contain connected components $\mathcal{C}_1$ and $\mathcal{C}_2$, respectively, of size at least $\beta N^{\alpha/2}$. 
Then
\[
\Prob\Big(\mathcal{C}_1\text{\emph{ connected by an open edge to }} \mathcal{C}_2
\;\Big|\; |\mathcal{C}_1|,|\mathcal{C}_2|\geq\beta N^{\alpha/2}\Big)
\geq 1-\exp\left(-\frac{\lambda(3d)^{-\alpha}\beta^2c^{2/{(\tau-1})}}{k^\alpha}\right).
\]
\end{lemma}
\proof
Since we assumed \eqref{lowerBoundWeights}, all weights are at least $c^{1/(\tau-1)}$. By Lemma \ref{maxDistance}, we get for arbitrarily chosen vertices $u\in\mathcal{C}_1, v\in\mathcal{C}_2$ that 
\[
\Prob(\{u,v\}\text{ is closed}) \leq \exp\left(-\lambda(3d)^{-\alpha}c^{2/{(\tau-1})}\frac{1}{kN^\alpha}\right).
\]
Since both clusters contain at least $\beta N^{\alpha/2}$ vertices, there are at least $\beta^2 N^\alpha$ possible edges. We obtain 
\begin{align*}
\Prob\Big(\mathcal{C}_1 &\text{ not connected by an open edge to } \mathcal{C}_2
\;\Big|\; |\mathcal{C}_1|,|\mathcal{C}_2|\geq\beta N^{\alpha/2}\Big)\\
&\leq \exp\left(-\lambda(3d)^{-\alpha}c^{2/{(\tau-1})}\frac{1}{kN^\alpha}\right)^{\beta^2 N^{\alpha}}\\
&=\exp\left(-\lambda(3d)^{-\alpha}\beta^2c^{2/{(\tau-1})}\frac{1}{k^\alpha}\right). \qed
\end{align*}

Note that this is the $\alpha\in(d,2d)$-counterpart of Lemma \ref{transientLemmaConnectivity}.
\proof[Proof of Theorem \ref{transientRandomThm}]
The previous lemmas readily imply the result for sufficiently large $\lambda$, and we are left to extend this to all $\lambda>\lambda_c$, which we achieve via coarse-graining. 

When $\gamma<2$, let $\lambda_0$ and $\mu_0$ be the values that we obtain from Proposition \ref{lemmaTransientLargeLambda}. To apply the proposition for all $\lambda > \lambda_c = 0$, we partition $\Zd$ into (hyper)cubes of side length $N$ (for some large $N$ to be determined below), which we call $N$-boxes. 
In every $N$-box we identify the vertex in it with maximum weight and call it the dominant vertex. We now choose $\beta$ large enough so that $\lambda\beta^2(3d)^{-\alpha}>\lambda_0$. 
Second, we call those $N$-boxes good that contain a vertex with weight at least $\beta N^{\alpha/2}$. We choose $N$ large enough so that the probability that an $N$-box is good is larger than $\mu_0$ using Lemma \ref{transientLemmaBigDegrees}. 
Lemma \ref{transientLemmaConnectivity} implies that the probability that the dominant vertices in two good $N$-boxes, being \emph{$k$ boxes away from each other}, are connected, is bounded from below by $\P_{(\lambda_0, W'')}(\{v_1, v_2\} \text{ is open})$,
where $v_1, v_2 \in \Zd$ such that $|v_1 -v_2|=k$, and where $W''_1,W''_2$ are i.i.d.\ distributed according to \eqref{standardPowerLaw}. Thus, the status of the edges between dominant vertices in good $N$-boxes stochastically dominates an SFP model on $\Zd$ with parameters $\alpha$, $\lambda_0$ and weight-law $W''$, combined with a site percolation of intensity $\mu_0$, exactly as described in Proposition \ref{lemmaTransientLargeLambda}.
We now apply Proposition~\ref{lemmaTransientLargeLambda} to obtain the result for the case $\gamma < 2$. 

The case $\alpha \in (d,2d)$ is analogous: When $\alpha\in(d,2d)$, let $\lambda_1$ and $\mu_1$ be the values that we obtain from Lemma \ref{sitebondlongrangeLemma}. To apply the lemma, we partition $\Zd$ into $N$-boxes again. Choose $\beta$ large enough, using Lemma~\ref{largeConnectivity}, such that two $N$-boxes being $k$-boxes away from each other, having clusters $\mathcal{C}_1$ and $\mathcal{C}_2$ with size at least $\beta N^{\alpha/2}$ are connected by an open edge between $\mathcal{C}_1$ and $\mathcal{C}_2$ with probability at least
$1-\exp\left(-{\lambda_1}/{k^\alpha}\right)$.
Call the $N$-boxes that contain a cluster of size at least $\beta N^{\alpha/2}$ the good boxes. Choose $N$ large enough so that the probability that an $N$-box is good is larger than $\mu_1$, using Lemma \ref{clustersizeTransient}. We thus find that the status of the edges between the dominant vertices of good $N$-boxes stochastically dominates an LRP model on $\Zd$ with independent edge probabilities $p_{x,y} = 1 - \exp(-\lambda_1/ |x-y|^{\alpha})$, combined with a site-percolation of intensity $\mu_1$. An application of Lemma \ref{sitebondlongrangeLemma} thus obtains the result for the case $\alpha \in (d,2d)$.

We conclude that in both cases we found a subgraph of the infinite cluster on which the random walk is transient, and hence the random walk is transient on the infinite cluster itself, cf.\ \cite[Section 9]{LectureNotesPeres}. \qed

\newpage
\subsection*{Recurrence proof}
We verify that we can apply the following lemma.

\begin{lemma}[Berger {\cite[Theorem 3.10]{BergerTransience}}]\label{lemma2Drecurrent}
Let $d=2, \alpha\geq2d = 4$ and let $(P_{i,j})_{i,j\in\mathbb{N}}$ be a family of probabilities, such that 
\[
\limsup_{i,j\rightarrow\infty} \frac{P_{i,j}}{(i+j)^{-4}}<\infty.
\]
Consider a shift invariant percolation model on $\mathbb{Z}^2$ on which the bond between $(x_1,y_1)$ and $(x_2,y_2)$ is open with marginal probability $P_{|x_1-x_2|,|y_1-y_2|}$. If there exists an infinite cluster, then this cluster is recurrent.
\end{lemma}

To bound the marginal probabilities we need a bound on the expectation of the product of the weights.
\begin{lemma}\label{upperboundProductWeights}
Assume that the weight-distribution satisfies \eqref{upperBoundWeights}. Let  $W_1,W_2$ be two independent copies of the random variable $W$. 

If $\tau>2$, there exists a constant $C>0$, such that for $u\geq 1$:
\[
\mathbb{E}\left[\left(W_1W_2/u\right)\wedge1\right]\leq C u^{-1}.
\]

If $\tau\leq 2$, then there exists a constant $C>0$, such that
\[
\mathbb{E}\left[\left(W_1W_2/u\right)\wedge1\right]\leq C \log(u)u^{-(\tau-1)}.
\]
\end{lemma}

\proof
The proof for $\tau>2$ is straightforward. Observe that $\mathbb{E}[W]<\infty$, hence
\[\mathbb{E}\left[\left(W_1W_2/u\right)\wedge1\right]\leq \mathbb{E}[W]^2/u.\]

For $\tau\leq 2$, we prove the claim for weights that satisfy \eqref{e:Wprime} for some $c\geq 1$. Lemma \ref{obs:increasing} then implies that the claim holds for weights that satisfy \eqref{upperBoundWeights}.

Let $H(u)$ denote the distribution function of $W_1W_2$. From \cite[Lemma 4.3]{DeijfenScaleFree} we have for some $C'>0$ that
\[
1-H(u)\leq C'\log(u)u^{-(\tau-1)}.
\]
By 
\[
\int\limits_1^uvdH(v)\leq \int\limits_1^u (1-H(v))dv,
\]
we obtain the result
\[\begin{split}
\mathbb{E}\left[\left(W_1W_2/u\right)\wedge1\right]&=1-H(u)+\frac{1}{u}\int_1^u vdH(v)\\
&\leq 1-H(u)+\frac{1}{u}\int\limits_1^u (1-H(v))dv\\
&\leq C'\log(u)u^{-(\tau-1)}+\frac{C'}{u}\int\limits_1^u\log(v)v^{-(\tau-1)}dv\\
&\leq C'\log(u)u^{-(\tau-1)}+\frac{C''}{u}\log(u)u^{-(\tau-2)}\\
&\leq C \log(u)u^{-(\tau-1)}. \qed
\end{split}
\]

\proof[Proof of Theorem \ref{recurrentRandomWalk2D}]
Observe that the scale-free percolation measure $\Prob_{(\lambda, W)}$ is indeed shift invariant.
 According to Lemma \ref{lemma2Drecurrent}, we need to prove that
\[
\limsup_{k\rightarrow\infty} k^4 \Prob\left(\{(0,0),(i,j)\}\text{ is open}\right)<\infty
\]
whenever $|i|+|j|=k$. For convenience, we treat only $(i,j)=(k,0)$, the other cases follow analogously. 
Lemma \ref{upperboundProductWeights} and the bound $1-\exp(-x)\leq x$ give 
\begin{align*}
\limsup_{k\rightarrow\infty} k^4 \Prob\left(\{(0,0),(k,0)\}\text{ is open}\right) &= \limsup_{k\rightarrow\infty} k^4\mathbb{E}\left[1-e^{-\lambda\frac{W_{(k,0)}W_{(0,0)}}{k^\alpha}}\right]\\
&\leq \limsup_{k\rightarrow\infty} k^4\mathbb{E}\left[1-e^{-\lambda\frac{W_{(k,0)}W_{(0,0)}}{k^\alpha}}\right]\\
&\leq \limsup_{k\rightarrow\infty} k^4\mathbb{E}\left[\left(\lambda\frac{W_{(k,0)}W_{(0,0)}}{k^\alpha}\right)\wedge1\right].
\end{align*}
For $\tau>2$, recall that $\alpha\geq 2d = 4$, and therefore
\begin{equation*}
\limsup_{k\rightarrow\infty} k^4\mathbb{E}\left[\left(\lambda\frac{W_{(k,0)}W_0}{k^\alpha}\right)\wedge1\right] \leq \limsup_{k\rightarrow\infty}C \lambda k^{4-\alpha}\leq C \lambda<\infty.
\end{equation*}
For $\tau\leq 2$ and $\gamma>2$, 
\begin{align*}
\limsup_{k\rightarrow\infty} k^4\mathbb{E}\left[\lambda\frac{W_{(k,0)}W_0}{k^\alpha}\wedge1\right]&\leq \limsup_{k\rightarrow\infty}k^4C \lambda^{\tau-1} \log(k^\alpha/\lambda)k^{-\alpha(\tau-1)}\\
&\leq \limsup_{k\rightarrow\infty} C'\log(k)k^{4-\alpha(\tau-1)}\\
&=\limsup_{k\rightarrow\infty} C' \log(k)k^{4-2\gamma}\\
&=0. \qed
\end{align*}

\section{Hierarchical clustering: proof of Theorem \ref{thm:trees}}\label{SectionClusters}
For the proof of Theorem \ref{thm:trees} we largely use the same two steps as for the proof of Theorem~\ref{transientRandomThm}. First we prove the result for large values of $\lambda$ on an SFP model combined with an i.i.d.\ Bernoulli site percolation, but only along a sequence $\{m_n\}_{n=1}^\infty$ diverging to infinity. Then, we prove the result in this SFP model for all sufficiently large $m$. Lastly, by a coarse graining argument, we extend the result to hold for all $\lambda>\lambda_c=0$. 

We assume the weights satisfy \eqref{e:Wprime}. Lemma \ref{obs:increasing} then implies that the claim holds for weights that satisfy \eqref{lowerBoundWeights}.
We start with stating and proving a proposition in which we consider a site-percolated version of SFP.
\begin{proposition}\label{lemLargestCluster}
Consider SFP on $\Zd$ with $d\ge1$, $\gamma<2$, and weight distribution that satisfies \eqref{standardPowerLaw}.
Independently of this, perform an i.i.d.\ Bernoulli site percolation on the vertices of $\Zd$, colouring a vertex ``green'' with probability $\mu \in (0,1]$. 
Denote by $\Ccal_{\lambda,\mu}$ the (unique) infinite subgraph of the infinite scale-free percolation cluster induced by the green vertices. We call this the site-percolated SFP.  

There exist constants $\mu_1 <1,\lambda_1>0,K, \rho>0$ and $n_2 \in \mathbb{N}$, and sequence $\{m_n\}_{n=1}^\infty$ with $m_n \in \mathbb{N}$, such that for all $\xi$ satisfying 
\begin{equation}\label{eqXiDef}
	 0<\xi<\min\left(\frac{d(2-\gamma)}{\tau+1},\frac{d}{2}\left(\tau+2-\sqrt{(\tau+2)^2-4(2-\gamma)}\right)\right),
\end{equation}
the following hold:
\begin{enumerate}
\item
The probability that the site-percolated SFP configuration with parameters $\mu \ge \mu_1$ and $\lambda \ge \lambda_1$ contains a $(0, m_n, \rho,K)$-hierarchically clustered tree inside the box $[0,m_n-1]^d$, is bounded from below by 
\[
1-3\exp\left(-\rho m_n^\xi \right).
\]
\item Site-percolated SFP with $\lambda \ge \lambda_1$ and $\mu \ge \mu_1$ has an infinite component $\mathcal{C}_{\lambda, \mu}$ almost surely. $\mathcal{C}_{\lambda, \mu}$ contains a.s.\ an infinite, connected, cycle-free subgraph $\mathcal{T}_\infty$ such that 
removing an arbitrary edge yields a finite and infinite connected component and there exist $x\in\Zd$ and $m\ge 1$ such that the finite connected component is an $(x,m,\rho,K)$-hierarchically clustered tree.
\end{enumerate}
\end{proposition}

\proof
Let $D_1$ be a large integer to be determined later, and let $\{a_n\}_{n=1}^\infty$, with ${a_n \in (0,1]}$, be a sequence also to be determined later, such that
\begin{equation}\label{constraintsAk1}
\rho:=\prod_{n=1}^\infty a_n > 0.
\end{equation}
Let 
\[
	\xi'\in \left(\xi, \min\left(\frac{d(2-\gamma)}{\tau+1},\frac{d}{2}\left(\tau+2-\sqrt{(\tau+2)^2-4(2-\gamma)}\right)\right)\right).
\] 
The bound $\xi'<d\,\frac{2-\gamma}{\tau+1}$ implies 
\begin{equation}
	\label{eqDefZeta}
	\zeta:=\frac{d-\xi'}{(\alpha+\xi')(\tau-1)-(d-\xi')}>1.
\end{equation}
For all $n \ge 2$, let $D_n :=\left \lceil D_{n-1}^{\zeta} \right \rceil$,
so that we have the telescoping product
\begin{equation}\label{formsDn}
D_n \ge D_{n-1}^{\zeta} \ge D_1\prod_{k=1}^{n-1}D_k^{\zeta-1} \ge D_1^{\zeta^{n-1}},
\end{equation}
and let
\[
u_n :=\prod_{k=1}^n D_k^{\frac{d-\xi'}{\tau-1}} \qquad \text{ and } \qquad C_n := a_n D_n^d.
\]

We give a sequence of graphs for $\Zd$ by the same procedure as in Proposition~\ref{lemmaTransientLargeLambda}. 
We call the vertices the $0$-stage boxes.
		We partition the lattice $\Zd$ into boxes of side-length $D_1$, so that each box contains $D_1^d$ vertices, and call these the $1$-stage boxes.
	Iteratively, we group $D_{n-1}^d$ $(n-1)$-stage boxes into $n$-stage boxes, so that the $n$-stage boxes form a covering of $\Zd$ by translates of $[0,\prod_{k=1}^n D_k -1]^d$.

We call a $0$-stage box ``good'' if the associated vertex is green, and we call this vertex ``$0$-dominant''.
For every stage $n \geq 1$, we define rules for a box to be ``good'' or ``bad'' and for a vertex to be ``$n$-dominant'' depending only on the weights $W_x$ and colours of the vertices and on the edges of $\mathcal{C}_{\lambda,\mu}$ inside the box. This implies that disjoint boxes are good or bad independently of each other.

For $n\geq 1$ we inductively define that an $n$-stage box is good if the following three conditions hold:
\begin{enumerate}
\item[(E)] At least $C_n$ of the $(n-1)$-stage boxes it contains are good. 
\item[(F)] The maximum weight $(n-1)$-dominant vertex in one of its good $(n-1)$-stage boxes has weight at least $u_{n}$. Call this vertex $n$-dominant. (A vertex can be dominant for multiple stages.)
\item[(G)] All $(n-1)$-dominant vertices in the good $(n-1)$-stage boxes are connected to the $n$-dominant vertex by an edge in $\Ccal_{\lambda,\mu}$.
\end{enumerate} 

\begin{figure}
\centering
\includegraphics[keepaspectratio,width = .6\textwidth]{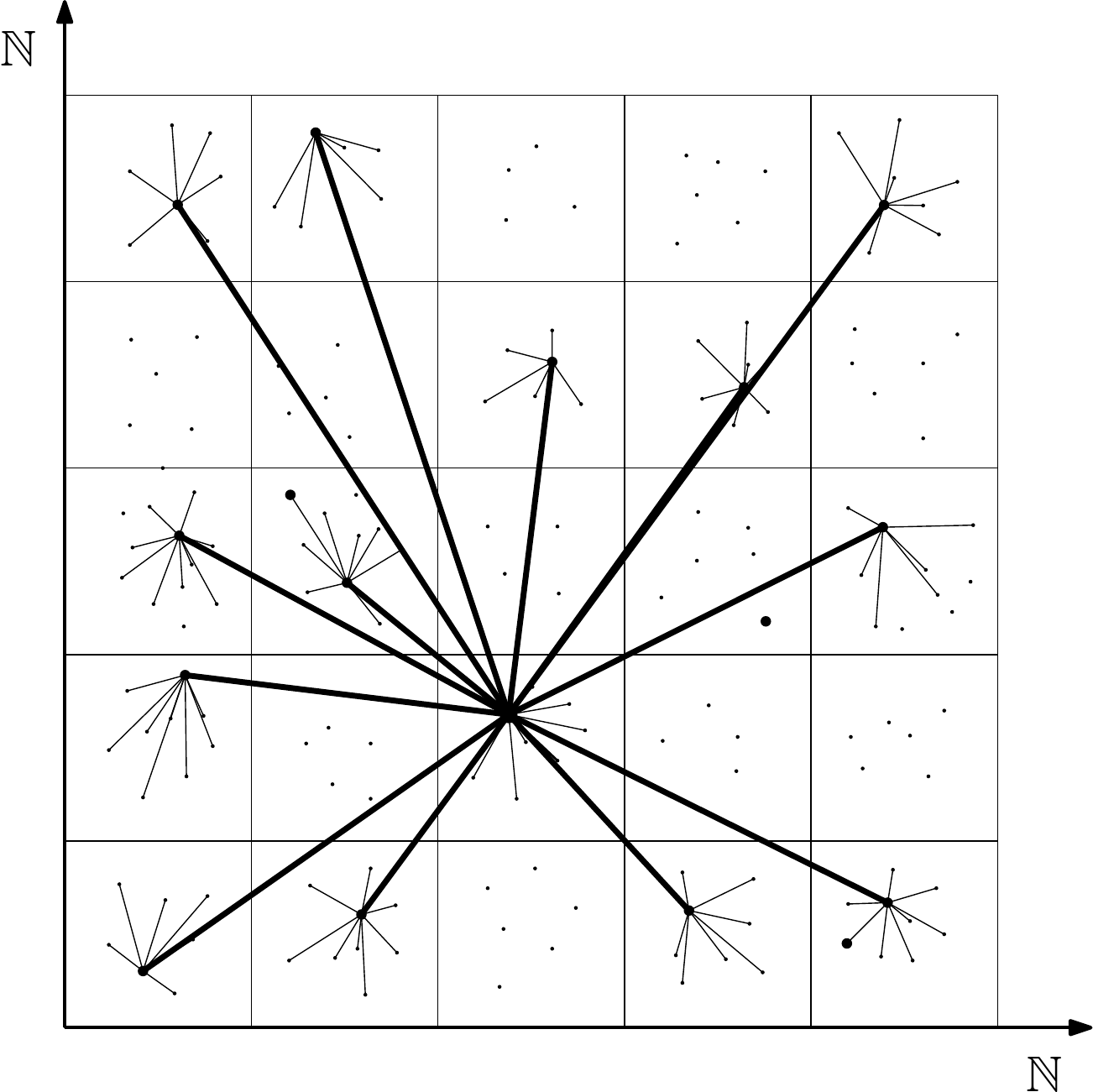}
\caption{Sketch of the renormalization in Proposition \ref{lemLargestCluster} for $d=2$.}\label{renormalizedFiniteBoxes}
\end{figure}

Define $E_n(v), F_n(v)$ and $G_n(v)$ to be the events that respectively (E), (F) and (G) hold for the $n$-stage box containing the vertex $v$. Denote by $L_n(v)$ the event that the $n$-stage box containing $v$ is good, i.e., $L_n(v) = E_n(v) \cap F_n(v) \cap G_n(v)$. To simplify notation we write $E_n:=E_n(0), F_n:=F_n(0), G_n:=G_n(0),$ and $L_n := L_n(0)$. 

On the event that $L_n$ occurs, we can construct a graph by the following procedure:
\begin{enumerate}
	\item Start with the set of all $0$-dominant vertices inside the $n$-stage box containing $v$. 
	\item For every $i \in \{0, \dots, n-1\}$ connect each $i$-dominant vertex $v$ to the $(i+1)$-dominant vertex inside the $(i+1)$-stage box that contains $v$ (unless this creates a self-loop).
\end{enumerate}
We  sequence $m_k$ and obtain a bound:
\begin{equation}\label{e:mprimedef}
m_n:=\prod_{k=1}^n  D_k \ge D_1^{\frac{\zeta^n-1}{\zeta-1}}.
\end{equation}
We claim that the constructed connected component of the $n$-dominant vertex in the $n$-stage box of $v$ is a tree that satisfies the conditions given in Definition \ref{def:hct}. Since \eqref{constraintsAk1} holds, the event $L_k$ implies that the intersection of the site-percolated SFP configuration and the cube $[0,m_k-1]^d$ contains a tree with at least
\[	\prod_{k=1}^n C_k = 
	\prod_{k=1}^n a_k  D_k^d 
	\ge\rho {m}_n^d   \]
vertices, which verifies Condition~(1) of Definition \ref{def:hct}. We obtain Condition~(2) for some $K>0$ because the box-size $m_n$ as well as the number of vertices in $[0,m_n-1]^d$ both grow double exponentially fast in $n$. Conditions~(3) and (4) follow straightforwardly from the construction. 
We therefore conclude that Proposition \ref{lemLargestCluster}(1) follows if we show that
\begin{equation}\label{eq Ln condition}
	\Prob(L_n^c)\leq 3\exp\left(-\rho\prod_{k=1}^nD_k^\xi\right),
\end{equation}
and that Proposition \ref{lemLargestCluster}(2) follows if we show that
\[
 	\Prob\left(\bigcap_{n=1}^\infty L_n\right)>0.
\]
The events $L_n$ are positively correlated, hence it is {sufficient} for assertion (2) of the theorem that 
$\prod_{n=1}^\infty \Prob(L_n)>0$,
which follows from \eqref{eq Ln condition}.  It thus remains to prove \eqref{eq Ln condition}.
\medskip

We bound
\begin{equation}\label{clusterLc}
\Prob(L_n^c)\leq \Prob(E_n^c)+\Prob(F_n^c \mid E_n)+\Prob(G_n^c \mid E_n\cap F_n),
\end{equation}
and analyze the terms separately.

We start with proving two bounds on $\Prob(F_n^c \mid E_n)$. The conditioning on $E_n$  gives that the sum of cluster sizes in the good $(n-1)$-stage boxes is at least $\prod_{k=1}^n C_k \ge  \prod_{k=1}^n a_k D_k^d$. Using Lemma \ref{maximumWeight} with $K_1=1, K_2=u_n$, the definition of $u_n$ yields that
\begin{equation}\label{clusterFc}
\Prob(F_n^c \mid E_n)\leq \exp\left(-u_n^{-(\tau-1)}\prod_{k=1}^na_kD_k^d\right)=\exp\left(-\prod_{k=1}^n a_k D_k^{\xi'}\right).
\end{equation}
Similarly, the conditioning on $E_n$ gives that there are at least $C_n$ vertices with weight at least $u_{n-1}$. Using Lemma \ref{maximumWeight} with $K_1=u_{n-1}, K_2=u_n$, the definition of $u_n$ yields that
\begin{equation}\label{clusterFc2}
\Prob(F_n^c \mid E_n)\leq \exp\left(-a_nD_n^d \left(D_n^{-\frac{d-\xi'}{\tau-1}}\right)^{\tau-1}\right) = \exp\left(-a_nD_n^{\xi'}\right).
\end{equation}
{Note that it is not a-priori clear which of these two bounds is better, since this may depend on the choice of $D_1, n, \tau, \xi'$ and the sequence $(a_k)$. Therefore we use both bounds.}

We move on to $\Prob(G_n^c \mid E_n \cap F_n)$.
For $n\ge2$ define  
\begin{equation}\label{eqDefBetaN}
\beta_n := u_n u_{n-1}\prod_{k=1}^n D_k^{-\alpha}.
\end{equation}
Since $D_n=\left\lceil D_{n-1}^\zeta\right\rceil$ by definition, it follows that
\[
D_{n-1} \ge \left(\frac{1}{2}D_n\right)^{1/\zeta}.
\]
Substituting the values of $u_n,u_{n-1},\zeta$, and \eqref{formsDn} gives 
\begin{align*}
\frac{\beta_n}{\beta_{n-1}}
&=\frac{u_n}{u_{n-2}}D_n^{-\alpha}
=D_n^{-\alpha+\frac{d-\xi'}{\tau-1}}D_{n-1}^{\frac{d-\xi'}{\tau-1}}
\ge \left(\frac{1}{2}\right)^{1/\zeta}D_n^{-\alpha+\frac{d-\xi'}{\tau-1}}\left(D_n^{\frac{(\alpha+\xi')(\tau-1)-(d-\xi)}{d-\xi'}}\right)^{\frac{d-\xi'}{\tau-1}}\\
&{=} \left(\frac{1}{2}\right)^{1/\zeta}D_n^{-\alpha+\frac{d-\xi'}{\tau-1}} D_n^{\alpha+\xi'-\frac{d-\xi'}{\tau-1}}
= \left(\frac{1}{2}\right)^{1/\zeta}D_n^{\xi'}.
\end{align*}
It follows that for some $c>0$,
\[
\beta_n \ge c\left(\frac{1}{2}\right)^{(n-1)/\zeta}\prod_{k=1}^nD_k^{\xi'}.
\]
The distance between two $(n-1)$-dominant vertices in the same $n$-stage box is at most $ d\prod_{k=1}^nD_k$.
There are at most $D_n^d$ good $(n-1)$-dominant vertices (having weight at least $u_{n-1}$). Recalling \eqref{eqDefBetaN}, by the union bound and the conditioning on $E_n$ and $F_n$, we thus obtain that
\begin{align}
\Prob(G_n^c \mid E_n \cap F_n)&\leq D_n^d\exp\left(-\lambda d^{-\alpha}u_nu_{n-1}\prod_{k=1}^n D_k^{-\alpha}\right)\nonumber\\
&= D_n^d\exp\left(-\lambda d^{-\alpha}\beta_n\right)\\
&\le \exp\left(d\log(D_n)-c\lambda d^{-\alpha}\left(\frac{1}{2}\right)^{(n-1)/\zeta}\prod_{k=1}^nD_k^{\xi'}\right).\nonumber
\end{align}
Since $D_n$ {grows} double exponentially fast and $\xi'>\xi>0$, we obtain for $n$ sufficiently large {(larger than $n_0$, say)}, that
\begin{equation}\label{clusterGc}
\Prob(G_n^c \mid E_n \cap F_n)\leq \exp\left(-\prod_{k=1}^nD_k^{\xi}\right).
\end{equation}

\medskip
We move on to $\Prob(E_n^c)$. All $(n-1)$-stage boxes are good independently of each other with probability $\Prob(L_{n-1})$. 
Let $X\sim\text{Bin}(D_n^d,\Prob(L_{n-1}))$. Then, 
\[
\Prob(E_n^c)=\Prob(X<C_n)=\Prob(X<a_nD_n^d).
\] 
As in \eqref{upperEc}, we apply Chernoff's bound and obtain
\[
\Prob(E_n^c)\leq\exp\left(-\frac{\left(\Prob(L_{n-1})-a_n\right)^2}{2\Prob(L_{n-1})}D_n^d\right),
\]
whenever ${0 < }a_n < \Prob(L_{n-1})$.
We now choose 
\begin{equation}\label{eqDefAn}
a_n:=\Prob(L_{n-1})\left(1-\sqrt{2}D_n^{-d/2}\prod_{k=1}^n D_k^{\xi'/2}\right).
\end{equation}
We will show below that $a_n >0$ for all $n \ge 1$ and that $\prod_{n=1}^\infty a_n >0$, as required by \eqref{constraintsAk1}. Assuming these inequalities we have
\begin{equation}\label{clusterEc}
\Prob(E_n^c)\leq\exp\left(-\Prob(L_{n-1})\prod_{k=1}^nD_k^{\xi'}\right), 
\end{equation}
so that using $\Prob(L_{n-1})>a_n>\rho$, $\xi'>\xi$, \eqref{clusterFc},  and \eqref{clusterGc} yields
\begin{align}\label{boundLn2}
\Prob(L_n^c)&\leq \exp\left(-\Prob(L_{n-1})\prod_{k=1}^nD_k^{\xi'}\right)+\exp\left(-\prod_{k=1}^n a_k\prod_{k=1}^n D_k^{\xi'}\right)+\exp\left(-\prod_{k=1}^nD_k^{\xi}\right)\nonumber\\
&\leq 3\exp\left(-\prod_{k=1}^n a_k\prod_{k=1}^n D_k^\xi\right), 
\end{align}
which gives the desired bound \eqref{eq Ln condition}.
For later reference we note that if instead we apply 
\eqref{clusterLc}, \eqref{clusterFc2}, and \eqref{clusterGc},  
then we obtain for $n$ sufficiently large 
\begin{equation}\label{boundLn1}
\Prob(L_n^c)\leq \exp\left(-\Prob(L_{n-1})\prod_{k=1}^nD_k^{\xi'}\right)+\exp\left(- a_n D_n^{\xi'}\right)+\exp\left(-\prod_{k=1}^nD_k^\xi\right).\end{equation}
\medskip

All that remains is to show that {$a_n >0$ for all $n \ge 1$ and that $\prod_{n=1}^\infty a_n >0$.} For positivity of $a_n$, it is by \eqref{eqDefAn} sufficient to show that for some $b>0$ and $D_1$ sufficiently large
\begin{equation}\label{dnDoubleExp}
D_n^{-d/2}\prod_{k=1}^n D_k^{\xi'/2} < \frac{1}{2\sqrt{2}}D_1^{-b\zeta^{n}}. 
\end{equation}
Since $\xi'>0, \zeta>1$,  $\prod_{k=1}^nD_k\leq \big(D_{n+1}/D_1)^{1/(\zeta-1)}$ and $D_{n+1}=\left\lceil D_n^\zeta\right\rceil\leq 2D_n^\zeta$ by \eqref{formsDn}, we obtain
\begin{align*}
D_n^{-d}\prod_{k=1}^n D_k^{\xi'} & \le D_n^{-d}\left(\frac{D_{n+1}}{D_1}\right)^{\xi'/(\zeta-1)}
= D_n^{-d}\left(\frac{\left\lceil D_{n}^\zeta\right\rceil}{D_1}\right)^{\xi'/(\zeta-1)}
 \le \left(\frac{2}{D_1}\right)^{\xi'/(\zeta-1)}D_n^{-d+\frac{\xi'\zeta}{(\zeta-1)}}.
\end{align*}
We show that 
\[
-d+\xi'\frac{\zeta}{\zeta-1} < 0.
\]
By definition $\xi'<\frac{d}{2}\left(\tau+2-\sqrt{(\tau+2)^2-4(2-\gamma)}\right)$. So we derive, after rearranging terms and dividing by $\sqrt{d}$, that
\[
-\frac{1}{\sqrt{d}}\xi'  + \frac{\sqrt{d}}{2}(\tau+2)  > \frac{\sqrt{d}}{2}\sqrt{(\tau+2)^2 - 4 (2-\gamma)}.
\]
Squaring and substituting $\gamma=\alpha(\tau-1)/d$ yield
\[
\frac1d\xi'^2-\xi'(\tau+2)+2d-\alpha(\tau-1)>0,
\]
so that after rearranging
\[
\frac1d(d-\xi')^2 > (\alpha+\xi')(\tau-1)-(d-\xi').
\]
Hence, we obtain the result after dividing by $(d-\xi')$ and, using \eqref{eqDefZeta}, substituting $\zeta$:
\[
1 - \frac{\xi'}{d} > \frac1\zeta, \] which can be inverted to give \[-d+\xi'\frac{\zeta}{\zeta-1} < 0.
\]
Hence,  there exists a constant $b>0$ such that, when we choose $D_1$ sufficiently large,
\[
D_n^{-d/2}\prod_{k=1}^n D_k^{\xi'/2} < \frac{1}{2\sqrt{2}}D_n^{-b\zeta}.
\]
By \eqref{formsDn} we have $D_n\geq D_1^{\zeta^{n-1}}$, so \eqref{dnDoubleExp} follows, and we may conclude that $a_n>0$.
\medskip

Observe that $\prod_{k=1}^\infty a_k>0$ if and only if $\prod_{k=1}^\infty \Prob(L_k) > 0$, since $a_n$ approaches $\Prob(L_n)$ double exponentially fast. Moreover, combining \eqref{eqDefAn} and \eqref{dnDoubleExp} gives
\[
a_n\geq \frac{1}{2}\Prob(L_{n-1}),
\]
so that, using \eqref{boundLn1}, we can bound
\begin{equation}\label{exp:clustLn}
\Prob(L_n^c)\leq \exp\left(-\Prob(L_{n-1})\prod_{k=1}^nD_k^{\xi'}\right)+\exp\left(- \frac{1}{2}\Prob(L_{n-1}) D_n^{\xi'}\right)+\exp\left(-\prod_{k=1}^nD_k^\xi\right).
\end{equation}
Define the sequence \[\ell_n:=1-(n+1)^{-3/2},\] and observe that
\begin{equation}\label{prodPositive2}
\prod_{n=1}^\infty \ell_n>0.
\end{equation}
For any fixed $n_1>n_0$, we can find $\lambda_0>0$ and $\mu_0<1$ such that 
 $\Prob(L_{n_1})\geq \ell_{n_1}$, because $L_{n_1}$ depends only on the weights and edges inside a \emph{finite} box.
Since $D_k$ grows double exponentially fast, we further bound  \eqref{exp:clustLn} for all $n > n_1$ by 
\begin{align}\label{recursiveBound2}
\Prob(L_n^c)&\leq \exp\left(-\left(1-\frac{1}{\sqrt{2}}\right)\prod_{k=1}^nD_k^{\xi'}\right) + \exp\left(-\frac12\left(1-\frac{1}{\sqrt{2}}\right)D_n^{\xi'}\right)+\exp\left(-\prod_{k=1}^nD_k^\xi\right)\nonumber\\
&\leq (n+1)^{-3/2}
=1-\ell_n.
\end{align}
We thus choose $n_1$ so large that the last bound in \eqref{recursiveBound2} holds.  
Using \eqref{prodPositive2}, \eqref{recursiveBound2} and $\Prob(L_n)>0$ for all $n$, yields that
\begin{equation*}
\prod_{n=1}^\infty\Prob(L_n)=\prod_{n=1}^{n_1}\Prob(L_n)\prod_{n=n_1+1}^\infty \Prob(L_n)\geq \prod_{n=1}^{n_1}\Prob(L_n)\prod_{n=n_1+1}^\infty \ell_n>0.
\end{equation*}

Recalling that, by \eqref{eqDefAn} and \eqref{dnDoubleExp}, this is equivalent to \eqref{constraintsAk1}, the bound \eqref{boundLn2} concludes the proof. \qed
\medskip

To prove Theorem~\ref{thm:trees} we need extend the above claims from the specific sequence $\{m_n\}_{n=1}^\infty$ to \emph{all} (sufficiently large) $m\in \mathbb{N}$. This extension is the content of the next lemma. After this lemma we extend the claim to hold for all $\lambda>0$ and weights following a power-law given by \eqref{e:Wprime}. 
\begin{lemma}\label{Lemma:Density}
Consider SFP on $\Zd$ with $d\ge 1, \gamma<2$, and weight distribution that satisfies \eqref{standardPowerLaw}. Independently of this, perform i.i.d.\ Bernoulli site percolation on the vertices of $\Zd$, colouring a vertex ``green'' with probability $\mu\in(0,1]$. Denote by $\mathcal{S}_{m, \lambda,\mu}$ the SFP-realization induced by the green vertices in the box $[0, m-1]^d$. We call this the site-percolated SFP. 

Then there exist a density $\rho>0$ and constants $\mu_1<1, \lambda_2>0$, and $K', m_0\in\mathbb N$, such that for $m\geq m_0$, and parameters $\lambda\ge \lambda_2, \mu\ge \mu_1$, the probability that $\mathcal{S}_{m, \lambda,\mu}$ with parameters contains a $(0, m,\rho,K')$-hierarchically clustered tree inside the box $[0,m-1]^d$ is bounded from below by 
\[
1-3\exp\left(-\rho m^\xi \right),
\] 
 whenever ${\xi<\min\left(\frac{d(2-\gamma)}{\tau+1},\frac{d}{2}\left(\tau+2-\sqrt{(\tau+2)^2-4(2-\gamma)}\right)\right)}$.
\end{lemma} 
\proof
Let the constants $\mu_1, \lambda_1, \xi', K$, the sequences $\{D_k\}, \{u_k\}, \{C_k\}$ and $\{m_k\}$ be as in Proposition~\ref{lemLargestCluster}, and $\zeta$ as in \eqref{eqDefZeta}. 
Assume $\mu\geq \mu_1, \lambda\geq \lambda_1$, and let $m$ be large enough (how large precisely will be determined in several steps). 
We define
\begin{equation}\label{def:nm}
n =\sup\{i:m_i\leq m\}, \qquad\text{ and }\qquad k=\left\lfloor \frac{m}{m_n}\right\rfloor,
\end{equation}
both depending on $m$, and note that $n\to\infty$ as $m\to\infty$. 
Partition the box $[0,km_n]^d$ into $k^d$ boxes of side-length $m_n$. We call these the $n$-boxes. Let $v^\ast$ be the vertex in $[0,km_n]^d$ with maximum weight. 

We use the same definition of good boxes and dominant vertices as in the proof of Proposition \ref{lemLargestCluster}. So in particular, an $n$-box is good if its $n$-dominant vertex has weight at least $u_n=m_n^{(d-\xi')/(\tau-1)}$.
We define $E$ to be the event that at least $\frac{1}{2}k^d$ $n$-boxes are good, $F$ the event that $W_{v^\ast}\geq (km_n)^{(d-\xi')/(\tau-1)}$, and $G$ the event that every good $n$-box's $n$-dominant vertex is connected by an open edge to $v^\ast$. Let $L=E\cap F\cap G$. 

Observe that the event $L$ implies indeed that there exists a $(0, m, \rho, K)$-hierarchically clustered tree. Properties~(1), (3) and (4) of Definition \ref{def:hct} readily follow from Proposition~\ref{lemLargestCluster}. For Property~(2), we observe that, for any $n\in\mathbb{N}$, the diameter of the constructed tree, after connecting the separate trees via $v^\ast$, only increases by 
2 in comparison to Proposition \ref{lemLargestCluster}. Hence there exists $K'>0$, such that Property~(2) is satisfied.
We bound 
\begin{equation}\label{eq:clusterGeneralBound}
\Prob(L^c)\leq \Prob(E^c)+\Prob(F^c\mid E)+\Prob(G^c \mid E\cap F),
\end{equation}
and analyze the three summands term by term. 

By our assumption \eqref{standardPowerLaw}, all $(km_n)^d$ vertices have weight at least 1. By Lemma \ref{maximumWeight},
\begin{equation}\label{eq:clusterBoundWeight}
	\Prob\left(F^c \mid E\right)\leq \exp(-(km_n)^\xi).
\end{equation}

The $\ell^1$-distance between two vertices in $[0,km_n]^d$ is bounded above by $dkm_n$, and the number of $n$-dominant vertices is at most $k^d$. Therefore, 
\begin{equation}\label{eqGcEFbound}
\Prob(G^c \mid E\cap F)\leq k^d \exp\left(-\lambda d^{-\alpha}m_n^{2\frac{d-\xi}{\tau-1}-\alpha}k^{\frac{d-\xi}{\tau-1}-\alpha} \right).
\end{equation}
We claim that
\begin{equation}\label{eq-complicatedbound}
m_n^{2\frac{d-\xi}{\tau-1}-\alpha}k^{\frac{d-\xi'}{\tau-1}-\alpha} \geq c\left(km_n\right)^{\xi'}, \qquad \text{for some constant } c>0.
\end{equation}
Indeed, by our choice of $n$ and $k$ in \eqref{def:nm}, $D_{n+1}=\left\lceil D_n^\zeta\right\rceil < D_n^\zeta+1$, and $m_n=\prod_{k=1}^nD_k$ by their definitions in \eqref{formsDn} and \eqref{e:mprimedef},  we have 
\[
\frac{k}{m_n^{\zeta-1}} \leq \frac{D_{n+1}}{m_n^{\zeta-1}} \le \frac{D_n^\zeta + 1}{D_n^{\zeta-1}\prod_{k=1}^{n-1} D_k^{\zeta-1}}= \frac{D_n}{m_{n-1}^{\zeta-1}} + \prod_{k=1}^n D_k^{1-\zeta}. 
\]
Iterating this gives 
\begin{align*}
\frac{k}{m_n^{\zeta-1}} \leq  \frac{D_{n+1}}{m_n^{\zeta-1}} &\leq \frac{D_n}{m_{n-1}^{\zeta-1}} +\prod_{k=1}^n D_k^{1-\zeta} \\
&\leq \frac{D_{n-1}}{m_{n-2}^{\zeta-1}} + \prod_{k=1}^{n-1} D_k^{1-\zeta}+\prod_{k=1}^n D_k^{1-\zeta}\leq \hdots\leq D_1 + \sum_{k=1}^n\left(\prod_{l=1}^{k}D_l\right)^{1-\zeta}.
\end{align*}
Recall that  by \eqref{eqDefZeta} $\zeta>1$, and by \eqref{formsDn} $\prod_{l=1}^kD_l \geq (1/D_1)D_1^{\zeta^k/(\zeta-1)}$, so
\begin{align}
\frac{k}{m_n^{\zeta-1}} & \le D_1 + \sum_{k=1}^n\left(\prod_{l=1}^{k}D_l\right)^{1-\zeta}
 \le D_1 + \sum_{k=1}^\infty\left(\prod_{l=1}^{k}D_l\right)^{1-\zeta}
\le D_1 + \sum_{k=1}^\infty D_1^{-\zeta^k} <\infty.
\end{align}
Hence, we can further estimate for some $c'>0$
\[
	k \leq c'm_n^{-1+\zeta} = c'm_n^{-1+\frac{d-\xi'}{(\alpha+\xi')(\tau-1)-(d-\xi')}} = c'm_n^{-1+\frac{d-\xi'}{\tau-1}\left(\alpha+\xi'-\frac{d-\xi'}{\tau-1}\right)^{-1}}, 
\]
so that 
\[
k^{\alpha+\xi'-\frac{d-\xi'}{\tau-1}} \leq {c'}^{\alpha+\xi'-\frac{d-\xi'}{\tau-1}}m_n^{-\left(\alpha+\xi'-\frac{d-\xi'}{\tau-1}\right)+\frac{d-\xi'}{\tau-1}},
\]
and finally
\[ 
{c'}^{\frac{d-\xi'}{\tau-1}-\xi'-\alpha}(km_n)^{\xi'}  \leq m_n^{2\frac{d-\xi'}{\tau-1}-\alpha}k^{\frac{d-\xi'}{\tau-1}-\alpha},
\]
from which \eqref{eq-complicatedbound} follows. 
Consequently, by \eqref{eqGcEFbound}, 
\begin{equation}\label{eq:clusterBoundConnections}
\Prob(G^c \mid  E\cap F)
\leq k^d \exp\left(-\lambda d^{-\alpha}{c'}^{\frac{d-\xi'}{\tau-1}-\xi'-\alpha}(km_n)^{\xi'}\right)
\leq \exp\left(-(km_n)^\xi\right), 
\end{equation}
where we choose $\lambda$ sufficiently large for the second bound (this determines the value of $\lambda_2$). 

It remains to bound $\Prob(E^c)$. By Proposition \ref{lemLargestCluster}, there exists $\rho>0$, such that for $n$ large, $n$-boxes are good independently of each other with probability at least $1-\exp(-\rho m_n^\xi)$. Let $X\sim \text{Bin}\big(k^d,1-\exp\big(-\rho m_n^\xi\big)\big)$. 
Writing out the binomial distribution, using  $\binom{n}{k}\leq n^k$ and $1-\exp(-x)\leq 1$, further bounding the sum by its maximum and the number of terms, we obtain
\begin{align*}
\Prob(E^c)
=\Prob\left(X<\frac{1}{2} k^d\right)
&=\sum_{l=0}^{\left\lfloor \frac{1}{2} k^d\right\rfloor}\binom{k^d}{l}\exp\left(-\rho m_n^\xi\right)^{k^d-l}\left(1-\exp\left(-\rho m_n^\xi\right)\right)^l\\
&\leq \left(\frac{1}{2} k^d+1\right)k^{\frac{1}{2}d k^d}\exp\left(-\frac{1}{2}\rho k^dm_n^\xi\right)\\
&=\exp\left(\log\left(\frac{1}{2} k^d+1\right)+\frac{1}{2}d k^d\log(k)-\frac{\rho}{2}k^dm_n^\xi\right).
\end{align*}
For any $\varepsilon$ with $0<\varepsilon<d-\xi$, we can take $m$ large enough so that  (using $k^{d-\xi}\ge1$)
\begin{equation}\label{eq:numberGood}
\Prob(E^c) \leq \exp\left(-\frac{\rho}{2^{1+\varepsilon}}(km_n)^{\xi}\right).
\end{equation}
Combining \eqref{eq:clusterGeneralBound}, \eqref{eq:clusterBoundWeight}, \eqref{eq:clusterBoundConnections}, and \eqref{eq:numberGood} gives that for $m$ sufficiently large,
\[\Prob(\mathcal{S}_{km_n}\text{ contains a }(0, km_n, \frac{\rho}{2},K)\text{-hierarchically clustered tree})\geq 1-3\exp\left(-\frac{\rho}{2^{1+\varepsilon}}(km_n)^\xi\right).\]
From the construction it follows that
\[
\frac12\leq\frac{{km_n}}{m}\leq1,
\]
so that for $m$ large
\begin{align*}
\Prob(\mathcal{S}_{m}&\text{ contains a }(0, m, \frac{\rho}{2^{d+1}},K')\text{-hierarchically clustered tree})\\
&\geq \Prob(\mathcal{S}_{km_n}\text{ contains a }(0, km_n, \frac{\rho}{2},K')\text{-hierarchically clustered tree})\\
&	\geq 1-3\exp\left(-\frac{\rho}{2^{1+\varepsilon}}\frac{m^\xi}{2^\xi}\right)
	\geq 1-\exp\left(-\frac{\rho}{2^{d+1}}m^\xi\right), 
\end{align*}
which finishes the proof. \qed
 
The last step is to extend the claim to hold for all $\lambda>0$ and weights following a power-law given by \eqref{e:Wprime}.
\begin{proof}[Proof of Theorem \ref{thm:trees}]
Recall that we consider SFP models with $1 < \gamma < 2$ and any $\lambda >0$, and that in this setting, $\lambda_c =0$ \cite{DeijfenScaleFree}.
Let $\mu_1$ and $\lambda_1$ be the values that we obtain from Lemma \ref{Lemma:Density}. To apply Lemma \ref{Lemma:Density}, we partition $\Zd$ into $N$-boxes. In every $N$-box we only consider the vertex with maximum weight and call it the dominant vertex.  Choose $\beta$ large enough, using Lemma \ref{transientLemmaConnectivity}, such that two $N$-boxes that are $k$-boxes apart, with dominant vertices $u_1$ and $u_2$ having weight at least $\beta N^{\alpha/2}$,  are connected by an open edge between $u_1$ and $u_2$ with probability at least
\[
\P_{(\lambda_1,W'')}(\{v_1, v_2\} \text{ is open}),
\]
for $v_1, v_2 \in \Zd$ such that $|v_1 -v_2| =k$ and weights $W''$ with law given by \eqref{e:Wprime}.
Define the $N$-boxes that contain a vertex with weight at least $\beta N^{\alpha/2}$ to be the good boxes. Choose $N$ large enough so that the probability that an $N$-box is good is larger than $\mu_1$ using Lemma~\ref{transientLemmaBigDegrees}.
Thus, the status of the edges between dominant vertices in good $N$-boxes in the SFP model with parameters $\alpha, \lambda$ and weight-law $W$ stochastically dominates an SFP model on $\Zd$ with parameters $\alpha$, $\lambda_1$ and weight-law $W''$ combined with a site percolation of intensity $\mu_1$, exactly as described in Lemma~\ref{Lemma:Density}. 
Let $K'$ and $\rho'$ be the constants  we obtain from Lemma~\ref{Lemma:Density}. Observe that
\[\frac{\left\lfloor\frac{m}{N}\right\rfloor N}{m}\geq \frac{1}{2}.\]
Then the assertions of Theorem \ref{thm:trees} follows if we set $K=K'$ and $\rho=\frac{\rho'}{2^dN^d}$.
\end{proof}

\subsection*{Acknowledgement}
JJ thanks the Erasmus+ programme for funding and the LMU Munich for hospitality during a three-month stay in autumn 2015. 
The work of TH is supported by the Netherlands Organisation for Scientific Research (NWO) through the Gravitation NETWORKS grant 024.002.003. 
We thank Remco van der Hofstad for suggesting the model to us and Thomas Beekenkamp for comments and corrections on an earlier version of the manuscript.

\end{document}